%% file: Drinfeld_Lemma-9.tex
\documentclass[12pt,one,english]{amsart}
\usepackage{mathrsfs}
\usepackage{mathabx}
\usepackage{graphicx}
\usepackage{enumerate}
\usepackage{dsfont}
\usepackage[english]{babel}
\usepackage{xcolor}
\usepackage[colorinlistoftodos]{todonotes}
\input{packages_and_functions.for_preamble.tex}

\usepackage{tikz-cd}
\usepackage[utf8]{inputenc}
\usepackage{mathtools}
\usepackage{mathrsfs}
\usepackage{amsmath,amssymb,amsfonts, mathabx}
\usepackage[all,cmtip]{xy}
\usepackage{enumitem}
\usepackage{extarrow}
\usepackage{enumerate}
\usepackage{fontenc}
\usepackage[T1]{fontenc}
\usepackage{graphicx}
\usepackage[colorlinks=true, linkcolor=blue]{hyperref}
\usepackage{latexsym}
\usepackage{mathrsfs}
\usepackage{mathtext}
\usepackage{tikz}
\usepackage{textcomp}
\usepackage{upgreek}
\usepackage{xy}
\usetikzlibrary{arrows}
\usetikzlibrary{matrix}
\usetikzlibrary{shapes}
\usetikzlibrary{snakes}
\usetikzlibrary{matrix}

\usepackage[
backend=biber,
style=alphabetic-verb,
]{biblatex}

\theoremstyle{plain}
\newtheorem{thm}{Theorem}[section]
\newtheorem*{thm*}{Theorem}
\newtheorem{lem}[thm]{Lemma}
\newtheorem{cor}[thm]{Corollary}
\newtheorem{prop}[thm]{Proposition}

\theoremstyle{remark}
\newtheorem{rmk}[thm]{Remark} 

\newtheorem*{rmk*}{Remark}

\theoremstyle{definition}
\newtheorem{defn}{Definition}[section] 

\newtheorem*{const*}{Construction}
\newtheorem{conv}[defn]{Conventions}

\theoremstyle{plain}
\newtheorem{thmI}{Theorem}
\newtheorem{thmII}{Theorem}

\newcommand{\drin}{\mathbf D_k}
\newcommand{\Fq}{{\mathbb{F}_q}}
\newcommand{\Emb}{{\textup{Emb}}}

\newcommand{\Et}{{\rm Et}}
\newcommand{\hvec}[1]{\begin{pmatrix} #1 \end{pmatrix}}
\newcommand{\alA}{\mathcal{A}}
\newcommand{\ind}{{\text{\sf ind}\hspace{.1ex}}}

\AtEveryBibitem{\clearlist{language}}

\begin{document}
\title{Drinfeld-Lau Descent over Fibered Categories}
\author{Valentina Di Proietto, Fabio Tonini, Lei Zhang}

\address{ 
Valentina Di Proietto\\
    University of Exeter\\
    College of Engeneering, Mathematics and Physical Sciences\\
Streatham Campus\\
    Exeter, EX4 4RN\\
United Kingdom }
\email{v.di-proietto@exeter.ac.uk}

\address{
Fabio Tonini\\
    Universit\'a degli Studi di Firenze\\
    Dipartimento di Matematica e Informatica 'Ulisse Dini'\\
    Viale Morgagni, 67/a\\ Firenze, 50134 Italy }
\email{fabio.tonini@unifi.it}

 \address{Lei Zhang\\
     Sun Yat-Sen University\\ School of Mathematics
     (Zhuhai)\\Zhuhai, Guangdong, P.~R.~China}
\email{cumt559@gmail.com}

\date{\today}

\input{packages_and_functions.tex}
\makeatletter 
\providecommand\@dotsep{5} 
\makeatother 

\setcounter{section}{0}
\maketitle

\begin{abstract} Let $X$ be a projective scheme over $\F_q$. Drinfeld-Lau
    descent asserts that for any algebraically closed field $k$ containing
    $\F_q$, the category of coherent sheaves on $X$ is equivalent to the
    category of coherent sheaves on $X_k\coloneqq X\times_{\F_q}k$ equipped
    with a Frobenius twist (the descent data). In this paper, we first
    generalize this descent theory to some general fibered
    categories (denoted by $\sX$), which includes proper algebraic stacks and affine gerbes.
    Then we replace the ``stack of coherent sheaves'' by some other stacks (denoted by
    $\sM$). In
    this case, the descent functor becomes the  pullback functor from $\sX$-sections of $\sM$ to
    $\sX_k$-sections of $\sM$ equipped with descent data.  We study when the
    pullback functor is faithful, fully faithful or an equivalence for
    different choices of $\sM$, and show that the pullback is an equivalence for $\sM$ being the stack of immersions, the stack of étale separated and quasi-compact maps, or any quasi-separated Deligne-Mumford stack with separated diagonal.
\end{abstract}

\section{Introduction}
We work over a finite field $\Fq$ of characteristic $p$, where
$q$ is a power of $p$ and fix an algebraically closed field $k$
over $\Fq$. We denote by $\phi_k\colon k\to k$ the power of the
absolute Frobenius corresponding to $q$, that is $a\mapsto a^q$. 
If $\stX$ is a scheme (or more generally a category fibered in groupoids) over $\Fq$ we set $\stX_k\coloneq \stX\times_\Fq k$, and, by abuse of notation,  denote by
$\phi_k$ also the arithmetic Frobenius of $\stX_k$, \emph{i.e.}
$\id_\stX\times \phi_k$.  

A very classical result of Lang asserts that the category $\Vect(\Fq)$ of finite $\Fq$-vector spaces is equivalent to the category $\Vect(k)^{\phi_k}$ of finite $k$-vector spaces with a $\phi_k$-linear automorphism.
Lang's result was then generalized to projective schemes:
\begin{thm}[Drinfeld-Lau descent \textup{\cite[Lemma 8.1.1]{Lau07}, \cite[Lemma 3]{Laf97}, \cite[Lemma 4.2.2]{Ked17}}]\label{Drinfeld-Lau Descent}
 If $\stX$ is a projective scheme over $\Fq$ then the pullback
along the projection $\stX_k\arr \stX$ induces an equivalence \[
   \Coh(\stX)\arr\Coh(\stX_k)^{\phi_k}
\]
from the category
of coherent sheaves on  $\stX$ to the category of coherent sheaves $\shF$ on $\stX_k$ equipped with an isomorphism $\phi_k^* \shF \to \shF$. 
\end{thm}
We refer to the previous result as to the Drinfeld-Lau descent. Drinfeld-Lau descent allows to pass informations from $\stX_k$ on $\stX$. The
purpose of this paper is to extend this form of descent phenomenon in several directions
but with the same spirit: interpret objects over $\stX$ as objects over
$\stX_k$ with a ``$\phi_k$-action''.
The more general setting in which we can formulate it is the following. 
\begin{itemize}
 \item As $\stX$ we  consider a category fibered in groupoids over $\Fq$, for instance general $\Fq$-schemes.
 \item The objects over $\stX$ are interpreted as objects of a category
     $\stM(\stX)$, where $\stM$ is a stack (for convenience in the étale
     topology) not necessarily in groupoids over $\Fq$. When $\stX$ is a category fibered in groupoids an object of $\stM(\stX)$ can be simply thought of as a base preserving functor $\stX \to \stM$.
 \item If $u\in \stM(\stX_k)$ then a $\phi_k$-action is an isomorphism $f\colon \phi_k^*(u)\to u$ in $\stM(\stX_k)$. A pair $(u,f)$ as before can  also be interpreted as a $2$-commutative diagram
     \begin{equation}\label{2-diagram}
\begin{tikzpicture}[xscale=2.0,yscale=-0.6,baseline=(current  bounding  box.center)]
  \node (A0_0) at (0, 0) {$\stX_k$};
  \node (A1_1) at (1, 1) {$\stM$};
  \node (A2_0) at (0, 2) {$\stX_k$};
  \path (A2_0) edge [->]node [auto,swap] {$\scriptstyle{u}$} (A1_1);
  \path (A0_0) edge [->]node [auto,swap] {$\scriptstyle{\phi_k}$} (A2_0);
  \path (A0_0) edge [->]node [auto] {$\scriptstyle{u}$} (A1_1);
\end{tikzpicture}
\end{equation}
We denote by $\stM(\drin(\stX))$ the category of those pairs $(u,f)$. (The
$\drin$ is for Drinfeld, while the notation $\stM(\drin(\sX))$ is used because
$\drin(\sX)$ is actually a quotient stack $[\stX_k/\Z]$, see \S \ref{sec: general remarks} for details.) 
\end{itemize}

Given $\stX$ and $\stM$ as before there is a well defined functor
\begin{equation}\label{the descent functor}
\alpha_{\stM,\stX} \colon \stM(\stX)\to \stM(\drin(\stX))\tag{$\star$}
\end{equation}
and the aim of this paper is to study this functor for different choices of $\sX$ and $\sM$. In Theorem \ref{Drinfeld-Lau Descent}, $\stM$ has taken to be the stack (not in groupoids) of finitely presented quasi-coherent sheaves, which we denote by $\QCoh_f$, while $\stX$ is a projective scheme over $\Fq$.
For non-proper schemes, the map
\(
\alpha_{\QCoh_f,\stX}\colon \QCoh_f(\stX)\to \QCoh_f(\drin (\stX)) 
\)
need not to be an equivalence. For example, if $\sX=\G_{m,\F_q}=\Spec(\F_q[T,\frac{1}{T}])$, then the pair $(\sO_{\sX_k},\sO_{\sX_k}\xrightarrow{\bullet\ T}\sO_{\sX_k})$ defines an element in $\QCoh_f(\drin(\sX))$, which is clearly not from $\QCoh_f(\sX)$.  We show that Drinfeld-Lau descent still works if $\stX$ is either a proper algebraic stack (e.g. a proper scheme) or an affine gerbe over $\Fq$. Actually, much more is true:

 \begin{thmI}[cf.~Theorem \ref{thm:the proper case2}]\label{special X}
     Let $\stX$ be a quasi-compact category fibered in groupoid over $\Fq$ {\rm(cf.~\S\ref{notation}~(2))}. Suppose that:
     \begin{enumerate}[label={\rm(\alph*)}]
       \item for all $\shF,\shG\in \QCoh_f(\stX_k)$, the $k$-vector space
    $\Hom_{\stX_k}(\shF,\shG)$ has finite dimension;
\item any quasi-coherent sheaf on $\stX$ is a quotient of a direct sum of sheaves in $\QCoh_f(\stX)$,
\end{enumerate}
     then the functor
     \(\alpha_{\stM,\stX}\) in \eqref{the descent functor}
is an equivalence in the following cases:
\begin{enumerate}[label={\rm (\arabic*)}]
    \item Drinfeld-Lau descent over fibered categories: $\stM=\QCoh_f$;
    \item $\stM$ is a quasi-compact algebraic stack with quasi-affine diagonal and there exists a representable fpqc covering $V\to \stX_k$ from a Noetherian scheme {\rm(}e.g. $\stX$ is of finite type over $\Fq$ or an affine gerbe{\rm)};
\item $\sM$ is a Noetherian algebraic stack with
quasi-affine diagonal;
\item $\sM$ is an affine gerbe over a
field.
\end{enumerate}
\end{thmI}

The condition (a) of Theorem \ref{special X} is similar to asking that $\sX_k$ is a pseudo-proper fibered category (cf.~\cite[Def. 7.1]{BV15}), however, these two conditions are different. The condition (b) is satisfied by quasi-compact and quasi-separated (qcqs) Deligne-Mumford stacks (cf.~\cite[Theorem A]{Rydh15}), Noetherian 
algebraic stacks (cf.~\cite[Prop. 15.4]{LMB00}, \cite{Rydh16}), and affine gerbes over $\F_q$.

Theorem \ref{special X}, (2), (3) and (4) are a consequence of Theorem
\ref{special X}, (1). The proof of this last result follows closely the
proof in \cite[Lemma 3]{Laf97}, which was, in the projective case, to interpret sheaves on $\stX_k$ as graded modules. The new idea is to replace sheafification of graded modules with a more general construction described in \cite{Ton20}.   

We also consider a variation of Drinfeld-Lau descent, asking for which kind of stacks $\stM$ the functor $\alpha_{\stM,X}$ is an equivalence for \emph{any} category fibered in groupoids $\stX$. 

\begin{thmII}\label{special M}\label{etale separated thm}
 Let $\stX$ be a category fibered in groupoids over $\Fq$ and let
 $\stM$ be a stack {\rm(}not necessarily in groupoids{\rm)} in the \'etale
 topology.
 \begin{enumerate}[label={\rm(\arabic*)}]
     \item The functor $\alpha_{\sM,\sX}$ is always faithful {\rm(cf.~Prop. \ref{main functor is faithful})}.
  \item If the $\Hom$-sheaves of $\stM$ are
 quasi-separated algebraic spaces then the functor
 \(
\alpha_{\stM,\stX}
\)
is fully faithful.\end{enumerate}
The functor \(
\alpha_{\stM,\stX}
\) is an equivalence when $\sM$ is one of the following stacks:
\begin{enumerate}[label={\rm(\Alph*)}]
    \item  a quasi-separated Deligne-Mumford stack with separated diagonal {\rm (}e.g. a
quasi-separated algebraic space{\rm )};
\item  the fppf stack $\Emb$ of embeddings {\rm (cf.~Theorem \ref{the case of closed immersion})};
\item the fpqc stack $\Et_s$ of
\'etale, quasi-compact and separated maps.
\end{enumerate}
\end{thmII}
Theorem \ref{special M}, (A) is not true if the stack $\sM$ is
algebraic but not Deligne-Mumford -- we have seen that, when
$\sX=\G_{m,\F_q}$ and $\sM=\Bi \G_{m,\F_q}$, the functor $\alpha_{\stM,\stX}$
is not an equivalence -- nor  
 if $\stM$ is Deligne-Mumford but not quasi-separated (e.g. if
 $\sX=\Spec(\F_q)$ and $\sM=\Bi \Z$, then it is easy to see that the
 $\Z$-torsor $\Spec(k)\to\drin(\F_q)=[k/\Z]$ does not come from a $\Z$-torsor
 $P\to \Spec(\F_q)$, because otherwise $P$ would be quasi-compact).
  Theorem \ref{etale separated thm}, (C) was known for the stack of finite \'etale
maps (cf.~\cite[Lemma 4.2.6]{Ked17}). 

Let us also briefly mention some
applications of our results.

\begin{itemize}
    \item As a consequence of Theorem \ref{special X}, (3) we show -- in Theorem \ref{affine gerbes are trivial} -- that any affine gerbe over $\F_q$ is trivial. This is previous known only for affine gerbes of finite type.
    \item We generalize Drinfeld's lemma (cf.~\cite[Theorem 2.1]{Dri80}, \cite[IV.2, Theorem 4]{Laf97}, \cite[Theorem 8.1.4]{Lau07}, \cite[Theorem 4.2.12]{Ked17}, \cite[Theorem 16.2.4]{SW20}) to connected algebraic stacks. The details of this will be carried out in a separated paper.
    \item We show by an example that the pro-étale fundamental group (resp. the geometric cover) does not satisfy Drinfeld's lemma (resp. Drinfeld-Lau descent) in a naive way (cf.~\S \ref{pro-étale}). 
\end{itemize}

\section{Notation}
\label{notation}
\begin{enumerate}
    \item By a stack over a ring $R$ we mean a stack (not
        necessarily in groupoids) for the fpqc topology over
        $\Sch/R$. If we want to consider a stack in a different topology we will specify the topology.

    \item We call a  category $\sX$ fibered in groupoids over a ring $R$
        \textit{quasi-compact} if it  admits a
        representable fpqc covering from an affine scheme.

    \item We call  category $\sX$ fibered in groupoids over a ring $R$ 
        \textit{quasi-separated} if the diagonal map
    is representable and quasi-compact.
    \item By \cite[3.1.2-3.1.3]{Vis05} there is an 2-equivalence between
        fibered categories and pseudo-functors. We will use these two notions interchangeably in
        this paper.
\end{enumerate}
\section{General Remarks}
\label{sec: general remarks}
Let $\stX$ be a category fibered in groupoids over $\Fq$ and $\stM$ be a stack in the \'etale topology over $\Fq$. In this section we reinterpret the map 
\[
\alpha_{\stM,\stX}\colon \stM(\stX)\to \stM(\drin(\stX))
\]
We define
\[
\drin(\Fq) = [\Spec k/\underline \Z]\text{\hspace{20pt}
    and\hspace{20pt}
}\drin(\stX)\coloneq\drin(\Fq)\times_\Fq \stX
\]
for a fibered category $\stX$ over $\Fq$, where $\underline \Z$ is the constant group scheme over $\Fq$ of $\Z$. The action of $\Z$ on $\Spec k$ is the one induced by the Frobenius $\phi_k$. Since  
\[
\underline \Z = \bigsqcup_{n\in \Z} \Spec \Fq \to \Spec \Fq
\]
is a covering in the \'etale (actually Zariski) topology, a $\underline
\Z$-torsor $P\to X$ for the fpqc topology over a scheme $X$ is automatically a
$\underline \Z$-torsors for the \'etale topology: $P$ become trivial after the
\'etale covering $P\to X$. This means that the quotient $[\Spec k/\underline
\Z]$ made with respect to the \'etale topology is a stack in the fpqc topology,
or, in other words, it coincides with quotient made with respect to the fpqc topology.  

One can check easily that the action of $\phi_k$ on $\Spec k$ is
free, so that $\drin(\Fq)$ is actually a sheaf (in the fpqc
topology). It is an algebraic space in the sense of
\cite[\href{https://stacks.math.columbia.edu/tag/025X}{025Y}]{stacks-project}, but not in the sense of
\cite[Definition 4.1]{LMB00}, because $\drin(\Fq)$ is not quasi-separated.

By definition of $\drin(\stX)$, the map $\stX_k\to \drin(\stX)$ is a $\underline \Z$-torsor and the corresponding action of $\Z$ on $\stX_k$ is given by the geometric Frobenius $\phi_k$.
If $\stX$ is just a sheaf in the \'etale topology we can conclude that $\drin(\stX)=[\stX_k/\Z]$ is the stack quotient for the \'etale topology.

In general there is a $2$-commutative diagram
\[
\begin{tikzpicture}[xscale=2.0,yscale=-0.6]
  \node (A0_0) at (0, 0) {$\stX_k$};
  \node (A1_1) at (1, 1) {$\drin(\stX)$};
  \node (A2_0) at (0, 2) {$\stX_k$};
  \path (A2_0) edge [->]node [auto,swap] {$\scriptstyle{u}$} (A1_1);
  \path (A0_0) edge [->]node [auto,swap] {$\scriptstyle{\phi_k}$} (A2_0);
  \path (A0_0) edge [->]node [auto] {$\scriptstyle{u}$} (A1_1);
\end{tikzpicture}
\]
and it induces a functor
\[
\Hom_\Fq(\drin(\stX),\stM) \to \stM(\drin(\stX))
\]
\begin{lem}
If $\stM$ is a stack in the \'etale topology then the above
functor is an equivalence.
\end{lem}
\begin{proof}
    For each $T\in \Sch/\F_q$ and each section $T\xrightarrow{x}
    \drin(\stX)$ the pullback $T'\to T$ of $\sX_k\to\drin(\sX)$
    along $x$ is a $\Z$-torsor, \emph{i.e.} $T=[T'/\Z]$. Given a diagram as
    in \eqref{2-diagram}, the universality of
    étale quotients then provides a unique (up to a unique isomorphism) 1-morphism
    $T\xrightarrow{y}\sM$. This defines the
    quasi-inverse functor
    $\stM(\drin(\stX))\to\Hom_\Fq(\drin(\stX),\stM)$.
\end{proof}

The above equivalence motivates the use of the symbol
$\stM(\drin(\stX))$. In this paper we identify functors
$\drin(\stX)\to \stM$ and objects of $\stM(\drin(\stX))$ when
$\sM$ is a stack in the \'etale topology. The
projection map $\drin(\stX)\to \stX$  induces a functor
\[
 \stM(\stX)\to \stM(\drin(\stX))
\]
which is easily checked to be the map $\alpha_{\stM,\stX}$
defined in the introduction.

In what follows we prove some general properties of the functor $\alpha_{\stM,\stX}$. We keep the notation from above.

We first observe a general procedure to reduce problems to categories fibered in groupoids.

\begin{lem}\label{from schemes to fibered categories} Assume that $\stM$ is a stack in the \'etale topology.
 If for all affine schemes $X$ over $\Fq$ the functor $\alpha_{\stM,X}\colon \stM(X)\to \stM(\drin(X))$ is fully faithful (resp. an equivalence) then $\alpha_{\stM,\stX} \colon \stM(\stX)\to \stM(\drin(\stX))$ is so for all categories $\stX$ fibered in groupoids over $\Fq$.
\end{lem}
\begin{proof}
    Consider the pseudo-functor $\sN\coloneq\underline{\Hom}_{\F_q}(\drin(\F_q),\sM)
    $ sending \[
        (T\in\Sch/\F_q)\xmapsto{\quad \quad}
\Hom_{T}(\drin(T),\sM_T)\simeq \Hom_{\F_q}(\drin(T),\sM)\]
Via the tautological equivalence $\sM\simeq
\underline{\Hom}_{\F_q}(\Spec(\F_q),\sM))$ one gets a canonical
1-morphism $\alpha\colon\sM\to\sN$.   
The functor 
    $\alpha_{\sM,\sX}$ is then identified as
    \[\alpha_{\sX}\colon\sM(\sX)\longrightarrow\sN(\sX)=\Hom_{\F_q}(\sX,\underline{\Hom}_{\F_q}(\drin(\F_q),\sM))\simeq\Hom_{\F_q}(\drin(\sX),\sM)
    \] 
    Now by \cite[Lemma 3.37, Prop. 3.36]{Vis05} if $\alpha_X$ is fully
faithful (resp. an equivalence) for each $X\in\Sch/\F_q$, then $\alpha$
is so, and  $\alpha_{\sM,\sX}$ is so too. Since $\sN$ is
clearly a stack in the étale topology, it is enough to consider 
only
affine schemes $X\in\Sch/\F_q$. 
%
\end{proof}

\begin{prop}\label{main functor is faithful}
If $\sM$ is a stack in the \'etale topology and a
    prestack in the fpqc topology, then the functor $\alpha_{\stM,\stX}\colon \stM(\stX)\to \stM(\drin(\stX)) $ is faithful.
\end{prop}
\begin{proof}
 This follows from the fact that $\drin(\sX)\to \sX$ is representable by an
 fpqc covering.
\end{proof}

Now we want to understand how to reduce the study of $\alpha_{\stM,\stX}$ locally on $\stX$.

\begin{conv}\label{situation scheme}
We consider the following data and assumptions:
\begin{itemize}
 \item a map  $\psi\colon V\to X$ of algebraic spaces;
 \item the stack $\stM$ is a prestack in the fpqc topology (e.g. an algebraic stack by \cite[\href{https://stacks.math.columbia.edu/tag/0APL}{0APL}]{stacks-project} );
 \item  the map $V\to X$  is an \'etale (resp. fppf,
fpqc) covering and $\stM$ is a stack in the \'etale (resp. fppf,
fpqc) topology.
\end{itemize}

 Consider the diagram
\[
\begin{tikzpicture}[xscale=4.0,yscale=-1.2]
  \node (A0_0) at (0, 0) {$\stM(X)$};
  \node (A0_1) at (1, 0) {$\stM(V)$};
  \node (A0_2) at (2, 0) {$\stM(V\times_X V)$};
  \node (A1_0) at (0, 1) {$\stM(\drin(X))$};
  \node (A1_1) at (1, 1) {$\stM(\drin(V))$};
  \node (A1_2) at (2, 1) {$\stM(\drin(V\times_X V))$};
  \path (A0_0) edge [->]node [auto] {$\scriptstyle{\psi^*}$} (A0_1);
  \path ([yshift=2.5pt]A0_1.east) edge [->]node [auto,yshift=1.0ex] {$\scriptstyle{}$} ([yshift=2.5pt]A0_2.west);
    \path ([yshift=-2.5pt]A0_1.east) edge [->]node [auto,yshift=1.0ex] {$\scriptstyle{}$} ([yshift=-2.5pt]A0_2.west);
  \path (A1_0) edge [->]node [auto] {$\scriptstyle{\psi^*}$} (A1_1);
  \path ([yshift=2.5pt]A1_1.east) edge [->]node [auto,yshift=1.0ex] {$\scriptstyle{}$} ([yshift=2.5pt]A1_2.west);
    \path ([yshift=-2.5pt]A1_1.east) edge [->]node [auto,yshift=1.0ex] {$\scriptstyle{}$} ([yshift=-2.5pt]A1_2.west); 
  \path (A0_2) edge [->]node [auto] {$\scriptstyle{\alpha_{\stM,V\times_X V}}$} (A1_2);
  \path (A0_0) edge [->]node [auto] {$\scriptstyle{\alpha_{\stM,X}}$} (A1_0);
  \path (A0_1) edge [->]node [auto] {$\scriptstyle{\alpha_{\stM,V}}$} (A1_1);
\end{tikzpicture}
\]

The two rows above are exact: indeed
$\drin(V\times_X V) \rightrightarrows \drin(V)\to\drin(X)$ is
just $V\times_X V \rightrightarrows V \to X$ base changed along
$\drin(\Fq)\to \Spec \Fq$. Notice moreover that the functors  $\alpha_{\stM,*}$ are faithful thanks to Prop. \ref{main functor is faithful}.
\end{conv}

The following results all follow by diagram chasing.

\begin{lem} Assume Conventions \ref{situation scheme}.
Let $\xi,\eta\in \stM(X)$ and $\omega\colon \alpha_{\stM,X}(\xi)\to \alpha_{\stM,X}(\eta)$ be a morphism. If $\psi^*(\omega)$ comes from a morphism $\psi^*(\xi)\to \psi^*(\eta)$ then $\omega$ comes from a morphism $\xi\to \eta$. 
\end{lem}
\begin{lem} \label{lemma of the essential surjectivity main lemma}  Assume Conventions \ref{situation scheme}.
 Assume moreover  that $\alpha_{\stM,V}$ and $\alpha_{\stM,V\times_X V}$ are fully faithful. Consider a $2$-commutative diagram as in the outer diagram of
 \[
\begin{tikzpicture}[xscale=2.0,yscale=-1.2]
  \node (A0_1) at (1, 0) {$V$};
  \node (A1_0) at (0, 1) {$\drin(V)$};
  \node (A1_2) at (2, 1) {$X$};
  \node (A1_3) at (3, 1) {$\stM$};
  \node (A2_1) at (1, 2) {$\drin(X)$};
  \path (A1_0) edge [->]node [auto] {$\scriptstyle{}$} (A2_1);
  \path (A1_0) edge [->]node [auto] {$\scriptstyle{}$} (A0_1);
  \path (A0_1) edge [->]node [auto] {$\scriptstyle{}$} (A1_2);
  \path (A2_1) edge [->]node [auto] {$\scriptstyle{}$} (A1_2);
  \path (A2_1) edge [->,bend left=30]node [auto] {$\scriptstyle{}$} (A1_3);
  \path (A0_1) edge [->,bend right=30]node [auto] {$\scriptstyle{}$} (A1_3);
  \path (A1_2) edge [->,dashed]node [auto] {$\scriptstyle{}$} (A1_3);
\end{tikzpicture}
\]
Then there exists a dashed arrow as in the above diagram and it is unique up to a unique isomorphism.
\end{lem}

We will often use the above lemma in the following form.
\begin{lem}\label{essential surjectivity main lemma}
 Assume Conventions \ref{situation scheme}.
Assume moreover that $\alpha_{\stM,Y}$ is fully faithful for all schemes $Y$ over $\Fq$. Consider commutative diagrams
  \[
  \begin{tikzpicture}[xscale=1.9,yscale=-1.2]
    \node (A0_0) at (0, 0) {$A$};
    \node (A0_1) at (1, 0) {$B$};
    \node (A0_2) at (2, 0) {$A$};
    \node (A0_3) at (3, 0) {$V$};
    \node (A1_0) at (0, 1) {$\drin(X)$};
    \node (A1_1) at (1, 1) {$\stM$};
    \node (A1_2) at (2, 1) {$\drin(X)$};
    \node (A1_3) at (3, 1) {$X$};
    \path (A0_0) edge [->]node [auto] {$\scriptstyle{}$} (A0_1);
    \path (A0_2) edge [->]node [auto] {$\scriptstyle{}$} (A0_3);
    \path (A1_0) edge [->]node [auto] {$\scriptstyle{}$} (A1_1);
    \path (A0_3) edge [->]node [auto] {$\scriptstyle{}$} (A1_3);
    \path (A0_2) edge [->]node [auto] {$\scriptstyle{}$} (A1_2);
    \path (A1_2) edge [->]node [auto] {$\scriptstyle{}$} (A1_3);
    \path (A0_1) edge [->]node [auto] {$\scriptstyle{}$} (A1_1);
    \path (A0_0) edge [->]node [auto] {$\scriptstyle{}$} (A1_0);
  \end{tikzpicture}
  \]
where $B$ is a sheaf and the second one is Cartesian, so that $A\simeq\drin(V)$. Assume moreover that the induced map $\drin(V)\to B$ factors as $\drin(V)\to V \to B$. Then there exists a $X\to \stM$, unique up to a unique isomorphism, fitting in the following commutative diagram
\[
\begin{tikzpicture}[xscale=2.0,yscale=-1.2]

  \node (A0_0) at (0, 0) {$A=\drin(V)$};
  \node (A0_1) at (1, 0) {$V$};
  \node (A0_2) at (2, 0) {$B$};
  \node (A1_0) at (0, 1) {$\drin(X)$};
  \node (A1_1) at (1, 1) {$X$};
  \node (A1_2) at (2, 1) {$\stM$};
  \path (A0_0) edge [->]node [auto] {$\scriptstyle{}$} (A0_1);
  \path (A0_1) edge [->]node [auto] {$\scriptstyle{}$} (A1_1);
  \path (A1_0) edge [->]node [auto] {$\scriptstyle{}$} (A1_1);
  \path (A1_1) edge [->]node [auto] {$\scriptstyle{}$} (A1_2);
  \path (A0_2) edge [->]node [auto] {$\scriptstyle{}$} (A1_2);
  \path (A0_0) edge [->]node [auto] {$\scriptstyle{}$} (A1_0);
  \path (A0_1) edge [->]node [auto] {$\scriptstyle{}$} (A0_2);
\end{tikzpicture}
\]
Moreover, in the above diagram, if the big square is Cartesian,
then so is the square on the right.
\end{lem}
\begin{proof}
 The first claim follows from Lemma \ref{lemma of the essential
 surjectivity main lemma}, so only the last one needs to be
 proved. 

 One checks easily that there is a Cartesian diagram:
 \[
 \begin{tikzpicture}[xscale=2.0,yscale=1.2,bmr/.pic={
\draw (0,0)--++(-90:2mm)--++(180:2mm);
}]
\path
(0,0)     node (F) {$V_k$}
+(0:1.5)  node (star) {$B_k$}
++(-90:1) node (X) {$X_k$}
+(0:1.5)  node (Y) {$\stM_k$};
\draw[->] (F)--(star);
\draw[->] (F)--(X);
\draw[->] (X)--(Y) node[midway,above,scale=.6]{};
\draw[->] (star)--(Y);
\end{tikzpicture} 
\]
which is the pullback of 
\[
 \begin{tikzpicture}[xscale=2.0,yscale=1.2,bmr/.pic={
\draw (0,0)--++(-90:2mm)--++(180:2mm);
}]
\path
(0,0)     node (F) {$V$}
+(0:1.5)  node (star) {$B$}
++(-90:1) node (X) {$X$}
+(0:1.5)  node (Y) {$\stM$};
\draw[->] (F)--(star);
\draw[->] (F)--(X);
\draw[->] (X)--(Y) node[midway,above,scale=.6]{};
\draw[->] (star)--(Y);
\end{tikzpicture} 
\]
along $\Spec k\arr\Spec\Fq$, whence the claim.
\end{proof}

\begin{cor}\label{local equivalences are goog}
  Assume Conventions \ref{situation scheme}. If $\alpha_{\stM,V}$ is fully faithful then so is $\alpha_{\stM,X}$. If  $\alpha_{\stM,V}$  and $\alpha_{\stM,V\times_X V}$ are equivalences then so is $\alpha_{\stM,X}$.
\end{cor}

\begin{rmk}\label{the affine case}
 If $\stX$ is a category fibered in groupoids over $\Fq$ and $\stM$ is an affine scheme then $\alpha_{\stM,\stX}$ is an isomorphism (of sets). Indeed by Lemma \ref{from schemes to fibered categories} we can assume that $\stX=X=\Spec A$ is affine. If $\stM=\Spec B$ then $\stM(X)=\Hom_\Fq(B,A)$ as algebras, while
 \[
 \stM(\drin(X))=\Hom_\Fq(B,(A\otimes_\Fq k)^{\phi_k})
 \]
 It is easily checked by choosing a basis of $A$ over $\Fq$ that
 the map $A\to (A\otimes_\Fq k)^{\phi_k}$ is an isomorphism,
 thus we see that $\alpha_{\stM,X}$ is an isomorphism.
\end{rmk}

\section{The case of immersions}

In this section we consider as $\stM$ the stack of immersions.

\begin{defn}
 We denote by $\Emb$ the fibered category of 
 immersions over $\Fq$. If $\stX$ is a category fibered in groupoids over $\Fq$, then
$\Emb(\stX)$ is the category of immersions in $\stX$, that is functors $\stZ\to \stX$ representable by immersions.
 \end{defn}

 The aim of this section is to prove the following.
\begin{thm}\label{the case of closed immersion}
 Let $\stX$ be a category fibered in groupoids. Then the functor
 \[
 \Emb(\stX)\to \Emb(\drin(\stX))
 \]
 is an equivalence of categories and it preserves quasi-compact (resp. open, closed) immersions
\end{thm}

%
%

We first show that $\Emb$ is a stack for the fppf topology.
\begin{prop}\label{Emb is a stack}
 The fibered category $\Emb$ is a stack in the fppf topology and a prestack in the fpqc topology. 
\end{prop}
\begin{proof}
 The fpqc-site $\Sch/\F_q$ is subcanonical, so objects in $\Emb$ are fpqc
 sheaves. In particular, $\Emb$ is a prestack in the fpqc topology. Effective descent for the fppf topology follows from \cite[\href{https://stacks.math.columbia.edu/tag/04SK}{04SK}]{stacks-project} and \cite[\href{https://stacks.math.columbia.edu/tag/02YM}{02YM}]{stacks-project}
%
%
\end{proof}

We need some preparatory lemmas. The following is called the
Moore determinant \cite[Corollary 1.3.7]{Goss98}. We present here a direct proof of the result.
\begin{lem}
The following relation holds in the ring $\Fq[x_0,\dots,x_r]$:
 \[
 \det 
 \begin{pmatrix}
x_0 & x_1 & \dots & x_r \\
x_0^q & x_1^q & \dots & x_r^q \\
\vdots & \vdots & \ddots & \vdots \\
x_0^{q^r} & x_1^{q^r} & \dots & x_r^{q^r}
\end{pmatrix}
= \omega \prod_{(a_0:\cdots:a_r)\in \PP^r(\Fq)}(a_0x_0 + \cdots + a_r x_r) 
 \]
where $\omega\in \Fq^*$ and, in the product on the right, one chooses a representative of an element of $\PP^r(\Fq)$.
\end{lem}
\begin{proof}
 Denote by $N=N(x_0,\dots,x_r)$ the matrix in the statement and by $F=F(x_0,\dots,x_r)$ its determinant. We are going to prove the equality on $\overline \Fq$. 
 
 Let $a=(a_0,\dots,a_r) \in \Fq^{r+1}$ be a non zero element and set 
 \[
 L_a = a_0 x_0 + \dots + a_rx_r \in \Fq[x_0,\dots,x_r]
 \]
 Notice that
 \[
 N(x_0,\dots,x_r) \cdotp a = (L_a, L_a^q,\dots,L_a^{q^r})
 \]
 where the right hand side should be thought of as a vertical vector. In particular it is clear that, if $(u_0,\dots,u_r)\in \overline \Fq^{r+1}$ is such that $L_a(u_0,\dots,u_r)=0$ then $N(u_0,\dots,u_r)\cdotp a = 0$ and therefore $F(u_0,\dots,u_r)=0$. In other words the zero locus $\{L_a=0\}\subseteq \overline \Fq^{r+1}$ is contained in the zero locus $\{F=0\}\subseteq \overline \Fq^{r+1}$. We can therefore conclude that $L_a$ divides $F$ in $\overline\Fq[x_0,\dots,x_r]$. 
 
 Notice that given $a,b\in \Fq^{r+1}$ we have that $L_a$ and $L_b$ generates the same ideal in $\overline \Fq[x_0,\dots,x_n]$ if and only if $a=\lambda b$ for some $\lambda \in \Fq$. Using the factorization into primes we can conclude that the product $P$ in the statement divides $F$. But
 \[
 \deg P = \#\PP^r(\Fq)=\frac{q^{r+1}-1}{q-1}= 1 + q + \cdots + q^r
 \]
 Using the inductive determinant formula it is easy to see that $\deg F = \deg P$, which ends the proof.
\end{proof}

\begin{cor}
\label{F-linear independent to k-linear ones}
 Let $L$ be a field over $\Fq$ and $\mu_0,\dots,\mu_r \in L$.
 Then $\mu_0,\dots,\mu_r$ are $\Fq$-linear independent if and only if the matrix
 \[
 \begin{pmatrix}
\mu_0 & \mu_1 & \dots & \mu_r \\
\mu_0^q & \mu_1^q & \dots & \mu_r^q \\
\vdots & \vdots & \ddots & \vdots \\
\mu_0^{q^r} & \mu_1^{q^r} & \dots & \mu_r^{q^r}
\end{pmatrix}
 \]
 is invertible.
\end{cor}

\begin{proof}[Proof of Theorem \ref{the case of closed immersion}]
The last claims hold because those properties of morphisms are local in the target for the fpqc topology (see \cite[\href{https://stacks.math.columbia.edu/tag/02L3}{02L3}]{stacks-project}, \cite[\href{https://stacks.math.columbia.edu/tag/02L6}{02L6}]{stacks-project}, \cite[\href{https://stacks.math.columbia.edu/tag/02L8}{02L8}]{stacks-project}).

 By Lemma \ref{from schemes to fibered categories} we can assume that $\stX=\Spec A$
 is affine. 

 1). Fully faithfulness. Faithfulness follows from Prop.  \ref{main functor is faithful} and Prop. \ref{Emb is a stack}. Thus let    $U_i\to \stX$, $i=1,2$ be immersions and $$\lambda_D\colon \drin(U_1)= 
 U_1\times_\stX\drin(\stX)\arr U_2\times_\stX\drin(\stX)=\drin(U_2)$$ be a map over $\drin(\stX)$. We have to show that $\lambda_D$ descends to a map 
 $\lambda \colon U_1\to U_2$.
 
 Consider the projection $p\colon U_1\times_\stX U_2 \to U_1$. By construction its base change along $\drin(\stX)\to \stX$ is an immersion with a section, thus an isomorphism. By fpqc descent it follows that also $p\colon U_1\times_\stX U_2 \to U_1$ is an isomorphism, which allows to define the map $\lambda$.
%
%
%

 2). Essential surjectivity. Let us first show how to conclude assuming the
 result being true for closed immersions.
Let $Z\to U \to \drin(A)$  be an
 immersion, where $Z\to U$ is a closed immersion and $U\subseteq \drin(A)$ is an open subset. Consider the reduced closed substack
 $C=\drin(A)\setminus U\to \drin(A)$, so that there exists a closed immersion $Q\to \Spec A$ inducing $C\to \drin(A)$. By construction $V=(\Spec
 A)\setminus Q \to \Spec A$ induces $U\to \drin(A)$. In particular $U=\drin(V)$. Applying again the result for closed immersion on $Z\to U=\drin(V)$ we get the result.
 
 Let us now focus on the case of closed immersions.  
 Since closed immersions form a substack of $\Emb$, using again Lemma \ref{from schemes to fibered categories} we just have to show that a closed embedding $Z\to
 \drin(A)$ comes from $A$. This closed embedding is given by an
 ideal $I$ of $A_k$ such that $\phi_k(I)=I$. We have to prove
 that $I$ is generated by elements of $A$. Let
 $\{\mu_i\}_{i\in J}$ be a
 basis of $k$ over $\Fq$. Given $f\in I\subseteq A_k$ we can write
 \[
     f=\sum_{r=0}^n \mu_{i_r} a_r
 \]
 where $a_r \in A$. Set \[M\coloneq
 \begin{pmatrix}
     \mu_{i_0} & \mu_{i_1} & \dots & \mu_{i_n} \\
     \mu_{i_0}^q & \mu_{i_1}^q & \dots & \mu_{i_n}^q \\
\vdots & \vdots & \ddots & \vdots \\
\mu_{i_0}^{q^n} & \mu_{i_1}^{q^n} & \dots & \mu_{i_n}^{q^n}
\end{pmatrix}
 \] 
then we have a matrix equation:
 \[
    \hvec{f\\\phi_k(f)\\\cdots\\\phi_k^{n}(f)}=
M\hvec{a_0\\a_1\\\cdots\\a_n} 
\]
 in which $M$ is invertible by Cor. \ref{F-linear independent to
 k-linear ones}.
 The relations obtained inverting $M$ in the equation above allow to conclude that
 \[
 \langle a_0, \dots, a_n \rangle_k= \langle f,\phi_k(f),\dots,\phi_k^{n}(f)\rangle_k \subseteq I 
 \]
 as $k$-vector spaces. This implies that 
 \[
     \{a_0,\dots,a_n\} \subseteq I\cap A
 \]
  hence $I\subseteq (I\cap A)A_k\subseteq I$, so $I=(I\cap A)A_k$
. \end{proof}

Using \ref{the case of closed immersion} we prove a variant of Theorem \ref{special M}, (1).

\begin{lem}\label{fully faithful for separated diagonal}
 Let $\stX$ be a category fibered in groupoids over $\Fq$ and
 $\stM$ be a stack in the \'etale topology with the following
 property: for all $\xi,\eta\in \stM(T)$ for a scheme $T$ the
 functor $\Homsh_\stM (\xi,\eta)\to T$ is a sheaf in the fpqc
 topology and its diagonal is an  immersion. Then the functor
 \[
 \alpha_{\stM,\stX}\colon \stM(\stX)\to \stM(\drin(\stX))
 \]
 is fully faithful.
\end{lem}
\begin{proof}
 By Lemma \ref{from schemes to fibered categories} we can assume that $\stX=X=\Spec A$ is affine. Let $\xi,\eta\in \stM(X)$ and set $H=\Homsh_\stM(\xi,\eta)\to X$. By Prop. \ref{main functor is faithful} we need to show that a commutative diagram 
   \[
  \begin{tikzpicture}[xscale=1.9,yscale=-1.2]
    \node (A0_1) at (1, 0) {$H$};
    \node (A1_0) at (0, 1) {$\drin(X)$};
    \node (A1_1) at (1, 1) {$X$};
    \path (A1_0) edge [->]node [auto] {$\scriptstyle{}$} (A0_1);
    \path (A1_0) edge [->]node [auto] {$\scriptstyle{}$} (A1_1);
    \path (A0_1) edge [->]node [auto] {$\scriptstyle{}$} (A1_1);
  \end{tikzpicture}
  \]
induces a section of $H\to X$. We have Cartesian diagrams of solid arrows
  \[
  \begin{tikzpicture}[xscale=1.9,yscale=-1.2]
    \node (A0_0) at (0, 0) {$X_k$};
    \node (A0_1) at (1, 0) {$\drin(X)$};
    \node (A0_2) at (2, 0) {$Z$};
    \node (A1_0) at (0, 1) {$H_k$};
    \node (A1_1) at (1, 1) {$\drin(H)$};
    \node (A1_2) at (2, 1) {$H$};
    \node (A2_0) at (0, 2) {$X_k$};
    \node (A2_1) at (1, 2) {$\drin(X)$};
    \node (A2_2) at (2, 2) {$X$};
    \path (A0_0) edge [->]node [auto] {$\scriptstyle{}$} (A0_1);
    \path (A2_0) edge [->]node [auto] {$\scriptstyle{}$} (A2_1);
    \path (A1_0) edge [->]node [auto] {$\scriptstyle{}$} (A1_1);
    \path (A0_1) edge [->,dashed]node [auto] {$\scriptstyle{}$} (A0_2);
    \path (A1_1) edge [->]node [auto] {$\scriptstyle{}$} (A1_2);
    \path (A0_2) edge [->,dashed]node [auto] {$\scriptstyle{}$} (A1_2);
    \path (A1_0) edge [->]node [auto] {$\scriptstyle{}$} (A2_0);
    \path (A1_1) edge [->]node [auto] {$\scriptstyle{}$} (A2_1);
    \path (A0_0) edge [->]node [auto] {$\scriptstyle{}$} (A1_0);
    \path (A0_1) edge [->]node [auto] {$\scriptstyle{}$} (A1_1);
    \path (A2_1) edge [->]node [auto] {$\scriptstyle{}$} (A2_2);
    \path (A1_2) edge [->]node [auto] {$\scriptstyle{}$} (A2_2);
  \end{tikzpicture}
  \]
  Since the diagonal of $H\to X$ is an immersion, the same property holds for the
  diagonal of $\drin(H)\to \drin(X)$. It follows that the
  section $\drin(X)\to \drin(H)$ is an immersion as well. Thanks to Theorem \ref{the case of
  closed immersion} we find an  immersion $Z\to H$ and a Cartesian diagram like in the
  above diagram. Since $Z$ is a sheaf in the fpqc topology and
  the map $Z\to X$ become an isomorphism after the fpqc base
  change $X_k\to X$ it follows that $Z\to X$ is an isomorphism
  as well. In conclusion the section $X\to Z \to H$ satisfies the requests.
\end{proof}

\section{The proper case}

The goal of this section is to prove Theorem \ref{special X}. We will actually work in a slightly more general situation than the one
in Theorem \ref{special X}. 
\begin{defn}
 Let $\stY$ be a quasi-compact category fibered in groupoids and $\shC\subseteq \QCoh(\stY)$ be a full subcategory. We say that $\shC$ generates $\QCoh(\stY)$ if all quasi-coherent sheaves on $\stY$ are a quotient of a direct sum of sheaves in $\shC$.
\end{defn}

\begin{conv} \label{situation of proper case} Let $\sX$ be a quasi-compact category fibered in groupoids over
$\F_q$.
Suppose that  
    \begin{itemize}
        \item     there is a full
    subcategory $\shC\subseteq\QCoh_f(\stX)$ generating $\QCoh(\stX)$
    which is stable under finite direct sums;
        \item for all $\shF\in
    \QCoh_f(\stX_k)$ the $k$-vector space
    $\Hom_{\stX_k}(\E\otimes k,\shF)$ has finite dimension for $\forall\,\E\in \shC$, where $\E\otimes k$ is the pullback of
    $\E$  along
$\pi\colon \stX_k \to \stX$.
\end{itemize}
\end{conv}

Here is the main theorem of this section:

\begin{thm}\label{thm:the proper case2}
 Suppose that $\stX$ satisfies Conventions \ref{situation of proper case}.
 Then
 $\alpha_{\sM,\sX}$ is an equivalence if
\begin{enumerate}
    \item $\stM=\QCoh_f$, or
    \item $\stM$ is a quasi-compact algebraic stack with quasi-affine diagonal
        and there exists a representable fpqc covering $V\to \stX_k$ from a
        Noetherian scheme, or
\item $\sM$ is a Noetherian algebraic stack with
quasi-affine diagonal, or
\item $\sM$ is an affine gerbe over a
field.
\end{enumerate}
\end{thm}

We start showing how the cases (2), (3), (4) follows from (1) of Theorem \ref{thm:the proper case2}.

\begin{lem}\label{ind coherent in quasi-coherent}
    Let $\stM$ be a quasi-compact and quasi-separated category fibered in
    groupoid. Consider two ind-objects
 $\{\sF_i\}_{i\in I},\, \{\sG_j\}_{j\in J}\in\ind(\QCoh(\sM))$, where
 \begin{enumerate}[label=(\arabic*)]
     \item either each $\sF_i\in\QCoh_f(\sM)$, 
     \item or each $\sF_i$ is of finite type, and all the transition maps in $\{\sG_j\}_{j\in J}$
 are injective,
 \end{enumerate}
  then the ``taking colimit'' functor induces an isomorphism
 of the $\Hom$s:
$$\Hom_{\ind(\QCoh(\stM))}(\{\sF_i\}_{i\in I},\{\sG_j\}_{j\in
J})\arr\Hom_{\QCoh(\stM)}(\varinjlim_{i\in
    I}\sF_i,\varinjlim_{j\in
    J}\sG_j)$$
 \end{lem}
\begin{proof}
Consider  $\sF\in \QCoh(\stM))$ and
    $\{\sG_j\}_{j\in J}\in \ind( \QCoh(\stM))$, where 
    \begin{enumerate}[label=(\arabic*)]
        \item  either $\sF\in \QCoh_f(\sM)$,
        \item  or $\sF$ is of finite type and the transition maps in
            $\{\sG_j\}_{j\in J}$ are injective.
    \end{enumerate}
    We have to show that the map 
    $$\Hom_{\ind(\QCoh(\stM))}(\sF,\{\sG_j\}_{j\in
    J})=\varinjlim_{j\in
J}\Hom_{\QCoh_f(\sM)}(\sF,\sG_j)\arr\Hom_{\QCoh(\stM)}(\sF,\varinjlim_{j\in
    J}\sG_j)$$
    is an isomorphism.
Let $\pi \colon U\to \stM$ be a fpqc covering from an affine scheme, set $R=U\times_\stX U$, which is a quasi-compact algebraic space, let $V\to R$ be an \'etale atlas from an affine scheme and set $\alpha,\beta \colon V\to U$ the induced maps. If $\shF,\shG\in \QCoh(\stM)$ then the following sequence is exact
 \[
     \Hom_{\QCoh(\stM)}(\shF,\shG) \to
     \Hom_{\QCoh(U)}(\pi^*\shF,\pi^*\sG)\rightrightarrows
 \Hom_{\QCoh(V)}(\alpha^*\pi^*\shF,\alpha^*\pi^*\sG)
 \]
 Using this we can reduce the problem to the case when $\sM$ is
 an affine scheme where the claim is obvious.
\end{proof}

\begin{lem}\label{general for proper and stacks}
Let $\stX$ and $\stM$ as follows:
\begin{enumerate}[label={\rm (\Alph*)}]
 \item $\stX$ is a category fibered in groupoids over $\Fq$ for which
 \[
 \alpha_{\QCoh_f,\stX}\colon \QCoh_f(\stX)\to \QCoh_f(\drin(\stX))
 \]
 is an equivalence;
 \item $\stM$ is a stack in groupoids over $\Fq$ for the fpqc (resp. fppf,
 \'etale) topology with
 quasi-affine diagonal and admitting a representable fpqc
 (resp. fppf, \'etale)  covering from an affine scheme;
 \item for any map
 $\drin(\sX)\xrightarrow{\psi}\sM$ and any
 quasi-coherent sheaf of algebras $\alA$ over $\stM$, the pullback $\psi^*\alA$ is the colimit of a ring object in $\ind(\QCoh_f(\drin(\stX)))$. 
\end{enumerate}
 Then the functor
 \[
 \alpha_{\stM,\stX}\colon \stM(\stX)\to \stM(\drin(\stX))
 \]
 is an equivalence. \end{lem}
\begin{proof}
By Lemma \ref{fully faithful for separated diagonal} we already know that $\alpha_{\stM,\stX}$ is fully faithful.

By hypothesis there is an fpqc (resp. fppf, \'etale) covering
$f\colon \tilde W=\Spec A\to \stM$. Since $\sM$ has quasi-affine
diagonal, it follows that $f$ is quasi-affine. Therefore,  if
$W\coloneq\Spec f_*\odi{\tilde W}$ then $\tilde W\to W$ is a
quasi-compact open immersion
\cite[\href{https://stacks.math.columbia.edu/tag/01SM}{01SM}]{stacks-project}. Now consider a functor $\psi\colon \drin(\stX)\to \stM$ and the Cartesian diagram 
   \[
  \begin{tikzpicture}[xscale=2.0,yscale=-1.2]
    \node (A0_0) at (0, 0) {$B$};
    \node (A0_1) at (1, 0) {$W$};
    \node (A1_0) at (0, 1) {$\drin(\stX)$};
    \node (A1_1) at (1, 1) {$\stM$};
    \path (A0_0) edge [->]node [auto] {$\scriptstyle{}$} (A0_1);
    \path (A1_0) edge [->]node [auto] {$\scriptstyle{\psi}$} (A1_1);
    \path (A0_1) edge [->]node [auto] {$\scriptstyle{}$} (A1_1);
    \path (A0_0) edge [->]node [auto] {$\scriptstyle{g}$} (A1_0);
  \end{tikzpicture}
  \]
  Then by hypothesis $\psi^*f_*\sO_{\tilde{W}}\simeq g_*\odi B$ is the colimit of a ring
  object in $\ind (\QCoh_f(\drin(\stX)))$. 
  We have a commutative diagram
  \[
  \begin{tikzpicture}[xscale=3.9,yscale=-1.2]
    \node (A0_0) at (0, 0) {$\ind(\QCoh_f(\stX))$};
    \node (A0_1) at (1, 0) {$\QCoh(\stX)$};
    \node (A1_0) at (0, 1) {$\ind(\QCoh_f(\drin(\stX)))$};
    \node (A1_1) at (1, 1) {$\QCoh(\drin(\stX))$};
    \path (A0_0) edge [->]node [auto] {$\scriptstyle{}$} (A0_1);
    \path (A0_0) edge [->]node [auto] {$\scriptstyle{}$} (A1_0);
    \path (A0_1) edge [->]node [auto] {$\scriptstyle{}$} (A1_1);
    \path (A1_0) edge [->]node [auto] {$\scriptstyle{}$} (A1_1);
  \end{tikzpicture}
  \]
Since the first vertical map is an equivalence by hypothesis, there is a ring object in $\ind(\QCoh_f(\stX))$ inducing a quasi-coherent sheaf of algebra on $\stX$ and therefore an affine map $V\to \stX$ with a Cartesian diagram
  \[
  \begin{tikzpicture}[xscale=2.0,yscale=-1.2]
    \node (A0_0) at (0, 0) {$B$};
    \node (A0_1) at (1, 0) {$V$};
    \node (A1_0) at (0, 1) {$\drin(\stX)$};
    \node (A1_1) at (1, 1) {$\stX$};
    \path (A0_0) edge [->]node [auto] {$\scriptstyle{}$} (A0_1);
    \path (A0_0) edge [->]node [auto] {$\scriptstyle{g}$} (A1_0);
    \path (A0_1) edge [->]node [auto] {$\scriptstyle{}$} (A1_1);
    \path (A1_0) edge [->]node [auto] {$\scriptstyle{}$} (A1_1);
  \end{tikzpicture}
  \]
In particular $B\simeq \drin(V)$. Consider the Cartesian diagrams on the left in
  \[
  \begin{tikzpicture}[xscale=1.5,yscale=-0.6]
    \node (A0_0) at (0, 0) {$\tilde B$};
    \node (A0_1) at (1, 0) {$\tilde W$};
    \node (A0_3) at (3, 0) {$\tilde B$};
    \node (A0_4) at (4, 0) {$\tilde V$};
    \node (A2_0) at (0, 2) {$\drin(V)$};
    \node (A2_1) at (1, 2) {$W$};
    \node (A2_2) at (2, 2) {$\then$};
    \node (A2_3) at (3, 2) {$\drin(V)$};
    \node (A2_4) at (4, 2) {$V$};
    \node (A4_0) at (0, 4) {$\drin(\sX)$};
    \node (A4_1) at (1, 4) {$\stM$};
    \node (A4_3) at (3, 4) {$\drin(\sX)$};
    \node (A4_4) at (4, 4) {$\sX$};
    \path (A4_0) edge [->]node [auto] {$\scriptstyle{}$} (A4_1);
    \path (A2_1) edge [->]node [auto] {$\scriptstyle{}$} (A4_1);
    \path (A0_3) edge [->]node [auto] {$\scriptstyle{}$} (A2_3);
    \path (A2_4) edge [->]node [auto] {$\scriptstyle{}$} (A4_4);
    \path (A0_0) edge [->]node [auto] {$\scriptstyle{}$} (A0_1);
    \path (A0_1) edge [->]node [auto] {$\scriptstyle{}$} (A2_1);
    \path (A0_4) edge [->]node [auto] {$\scriptstyle{}$} (A2_4);
    \path (A2_0) edge [->]node [auto] {$\scriptstyle{}$} (A4_0);
    \path (A2_3) edge [->]node [auto] {$\scriptstyle{}$} (A2_4);
    \path (A0_3) edge [->]node [auto] {$\scriptstyle{}$} (A0_4);
    \path (A4_3) edge [->]node [auto] {$\scriptstyle{}$} (A4_4);
    \path (A0_0) edge [->]node [auto] {$\scriptstyle{}$} (A2_0);
    \path (A2_0) edge [->]node [auto] {$\scriptstyle{}$} (A2_1);
    \path (A2_3) edge [->]node [auto] {$\scriptstyle{}$} (A4_3);
  \end{tikzpicture}
  \]
Applying Theorem \ref{the case of closed immersion} to the immersion $\tilde B\to \drin(V)$ we obtain the Cartesian diagrams on the right for some immersion $\tilde V\to V$. Since $\tilde V_k\to \stX_k$ is a base change of $\tilde{W}\to \stM$ which is an fpqc (resp. fppf, \'etale) covering, it follows that $\tilde V\to \stX$ is an fpqc (resp. fppf, \'etale) covering as well. From Lemma \ref{essential surjectivity main lemma} with $B=\tilde W$
 and Remark \ref{the affine case} we obtain the desired map $\sX\to \stM$.
\end{proof}
\begin{rmk}\label{filtered limit}Thanks to  Lemma \ref{ind coherent in
    quasi-coherent},
    condition (C) of Lemma \ref{general for proper and
    stacks} is satisfied when
\begin{itemize}
\item $\sM$ is a quasi-compact and quasi-separated algebraic
 stack and each of finite type sheaf on $\drin(\sX)$ is finitely
 presented. In this case, each object in
 $\QCoh(\sM)$ is the union of its quasi-coherent subsheaves of finite
 type (cf.~\cite{Rydh16}). 

\item $\sM$ is an affine gerbe over some field extension of $\F_q$. In this
    case,  each
    quasi-coherent sheaf is a filtered direct limit of vector
    bundles.
\end{itemize}
\end{rmk}

\begin{proof}[Proof of Theorem \ref{thm:the proper case2}, {(1)} $\implies$ 
    (2), (3), (4)]
 According to Lemma \ref{general for proper and stacks}  
 we just have to show that assuming (A), conditions (B), (C) of Lemma \ref{general for proper
 and stacks} hold in the situations of Theorem \ref{thm:the proper case2}, (2), (3), (4).
 Condition (B) for (2), (3) is due to
 \cite[\href{https://stacks.math.columbia.edu/tag/0GRH}{0GRH}]{stacks-project}  while (C) is
due to Remark \ref{filtered limit}.
\end{proof}

We now come back to the proof of Theorem \ref{thm:the proper case2}, (1). 
 The main result that we are going to use is the following, which is a hugely
 simplified version of \textup{\cite[Theorem B]{Ton20}}.

\begin{prop}\label{sheafification of linear
    functors}
Let $\stX$ be a quasi-compact category fibered in groupoid over
a ring $R$, $A$ be an $R$-algebra, $\shC\subseteq \QCoh_f(\stX)$ be a full
subcategory generating $\QCoh(\stX)$ and stable by finite direct sums and
denote by $\L_R(\shC,A)$ the category of contravariant $R$-linear functors
$\shC\to \Mod(A)$. Set $\sX_A\coloneqq \sX\times_RA$ and $u\colon
\sX_A\to\sX$ the projection.
Then the functor
\[
    \Gamma_{A/R}^{\sX}\colon \QCoh(\stX_A)\to \L_R(\shC,A)\comma \shF
    \longmapsto \Gamma_\shF=\Hom_{\stX_A}(u^*(-),\shF)
\]
is an equivalence onto the full subcategory
of functors which are exact on right exact sequences
in $\sC$ (left-exact functors).
\end{prop}

\begin{proof}[Proof (sketch)] The functor $\Gamma_{A/R}^{\sX}$ admits a
    left adjoint \(\int^{\sC}\colon
    \L_R(\sC,A)\to\QCoh(\sX_A)\) obtained using the ``coend
    construction'':
    Given $\Gamma\in\L_R(\sC,A)$, let $\Gamma_{\sE}$ denote the $A$-module
    associated with $\sE\in\sC$, then
    $\int^{\sC}\Gamma=\int^{\sE\in\sC}\Gamma_{\sE}\in\QCoh(\sX_A)$ is the cokernel of the
    map:\begin{equation*}\label{coend}
        \bigoplus_{\sE\xrightarrow{a}\sE'}(\Gamma_a\otimes\id_{u^*\sE}-\id_{\Gamma_{\sE'}}\otimes
        u^*a)\colon
        \bigoplus_{\sE\xrightarrow{a}\sE'}\Gamma_{\sE'}\otimes_Au^*\sE\arr
        \bigoplus_{\sE\in\sC}\Gamma_{\sE}\otimes_Au^*\sE
    \end{equation*}
    Indeed, one checks readily that there is a natural isomorphism:
    \[
        \Hom_{\QCoh(\sX_A)}(\int^{\sC}\Gamma,\sH)\xrightarrow{\ \cong\
        }\Hom_{\L_R(\sC,A)}(\Gamma,\Hom_{\sX_A}(u^*(-),\sH)),\,
        \forall\, \sH\in\QCoh(\sX_A)
    \]
    Actually, the restriction of $\int^{\sC}$ to the full subcategory of
    left-exact functors in $\L_{R}(\sC,A)$ is a quasi-inverse of
    $\Gamma_{A/R}^{\sX}$. To check this, one has to check
    \begin{enumerate}[label=(\alph*)]\item the counit
    $\epsilon$ of
    the adjunction is an isomorphism; \item the unit $\eta$, when restricted to the
left-exact functors, is an isomorphism.\end{enumerate} Since $\sX$ is quasi-compact,
    there is an fpqc-atlas $U=\Spec(B)\to \sX$. Fact (a) in this setting is somewhat tautological
    (cf.~\cite[Lemma 3.5, 3.6]{Ton20}). For example, for any
    $\sG\in\QCoh(\sX_A)$ the counit $\epsilon_{\sG}$
    is induced by the map
    \(\bigoplus_{\sE\in\sC}\Hom_{\sX_A}(u^*\sE,\sG)\otimes_A
        u^*\sE\to
    \sG\).
It is surjective because $u_*\sG\in \QCoh(\sX)$ and $\sC$ generates
$\QCoh(\sX)$.
    For (b), let
$\Gamma_{\sE}\xrightarrow{\eta_{\sE}}
    \Hom_{\sX_A}(u^*\sE,\int^{\sC}\Gamma)$ be the unit map evaluated
    at $\sE\in\sC$, and let $J_{B,\sC}$ be the category of pairs
    $(\sE,\psi)$ with $\sE\in\sC$, $\psi\in\sE(U)$. Then
    $\int^{\sC}\Gamma(U)=\displaystyle\varinjlim_{(\sE,\psi)\in
    J_{B,\sC}}\Gamma_{\sE}$ (cf.~\cite[Prop. 2.20]{Ton20}). For any $x\in
    \Ker(\eta_{\sE})$, $\psi\in\sE(U)$, the image of $x\otimes u^*\psi$ in
    $\int^{\sC}\Gamma(U)$ is
    equal to $\eta_{\sE}(x)(u^*\psi)=0$. Thus $\exists\ 
    (\sE,\psi)\xrightarrow{\mu}(\sE',\psi')\in J_{B,\sC}$ such that
    $\Gamma_{\mu}(x)=0$. In this way, 
     one finds a surjection $\oplus_{1\leq i\leq n}\mu_i\colon
    \bigoplus_{1\leq i\leq n}\sE_i\twoheadrightarrow \sE$, where $\sE_i\in\sC$,
    satisfying $\Gamma_{\mu_i}(x)=0\in\Gamma_{\sE_i}$. Then the left-exactness of
    $\Gamma$ implies
    that $x=0$, so $\eta$ is injective. To show that $\eta$ is
    surjective we take its cokernel $\Pi$ and want to show that
    $\Pi_{\sE}=0,\,\forall\,\sE\in\sC$. Writing 
    $\int^{\sC}\Gamma(U)$ as the colimit one sees that
    the sequence
    \(
        0\arr \int^{\sC}\Gamma \xrightarrow{\int^{\sC}\eta}
        \int^{\sC}\Gamma_{A/R}^{\sX}(\int^{\sC}\Gamma)\arr\int^{\sC}\Pi\arr
        0
    \) is exact.
    Note that $\int^{\sC}\eta$ has a retraction
    $\epsilon_{\int^{\sC}\Gamma}$ which is an isomorphism. This implies
    that $\int^{\sC}\Pi=0$. Now take any $x\in\Pi_{\sE}$, we have
    $\eta_{\sE}(x)=0$. Thus we can find a surjection
    $\oplus_{1\leq i\leq n}\mu_i\colon
    \bigoplus_{1\leq i\leq n}\sE_i\twoheadrightarrow \sE$ so that
    $\Gamma_{\mu_i}(x)=0$. A simple
    diagram chasing as in \cite[Lemma 4.7]{Ton20} shows that
    $x=0\in\Pi_{\sE}$. 
 \end{proof}

\begin{rmk}
    Applying Theorem \ref{Drinfeld-Lau Descent} to $\sX=\Spec(\F_q)$ we
    get an equivalence
    \(     \Vect(\Fq) \xrightarrow{\simeq}
    \Vect(k)^{\phi_k}=\Vect(\drin(\F_q))\)
whose quasi-inverse maps an object $(V,\sigma)$, where $V\in \Vect(k)$ and $\sigma\colon V\to V$ is a $\phi_k$-linear automorphism, to the $\Fq$-vector space 
 \(
 V^{\sigma}=\left\{ v\in V\mid \sigma(v)=v \right\}
 \).
 \end{rmk}

 \begin{proof}[Proof of Theorem \ref{thm:the proper case2},
     \textup{(1)}]  We want to define a
    functor $\QCoh_f(\drin(\stX))\to \QCoh_f(\stX)$ using
    Prop. \ref{sheafification of linear functors}. Consider an
    object in $\QCoh_f(\drin(\stX))$, which we think of as a
    pair $(\shF,\sigma)$ where $\shF\in \QCoh_f(\stX_k)$ and
    $\sigma\colon \phi_{k}^*\shF \to \shF$ is an isomorphism.
    With such a pair we can associate the $\Fq$-linear functor
    \(
    \Gamma_\shF=\Hom_{\stX_k}(-\otimes_\Fq k,\shF) \colon \shC\to \Vect(k)
    \) and a natural
    isomorphism $\Gamma_{\sigma}\colon\Gamma_{\phi_{k}^*\shF} \to \Gamma_{\shF}$. Notice that
\(
\Gamma_{\phi_{k}^*\shF}(\shE)=\phi_k^*(\Gamma_\shF(\shE)) \text{ for
}\forall\, \sE\in \shC 
\), so
 we conclude that $(\shF,\sigma)$ defines a functor $\Gamma_{(\shF,\sigma)}\colon \shC\to \Vect(\drin(\Fq))$. Via the  equivalence 
 \(\Vect(\Fq)\xrightarrow{\cong} \Vect(\drin(\Fq))\)
we obtain a functor
\(
    \Omega\colon \QCoh_f(\drin(\stX))\longrightarrow \L_{\F_q}(\shC,\Fq)
\)
sending
\(
(\sF,\sigma)\longmapsto \Omega_{(\sF,\sigma)}
\). There is also 
an isomorphism $\Omega_{(\sF,\sigma)}\otimes_\Fq k \xrightarrow{\cong}
\Gamma_\shF$ which is compatible with
$\id_{\Omega}\otimes \phi_k$ and $\Gamma_{\sigma}$. Notice moreover that by construction the composition
\begin{equation}\label{compare Omega and Gamma}
    \QCoh_f(\stX)\arr
    \QCoh_f(\drin(\stX))\xrightarrow{\ \Omega\ } \L_{\F_q}(\shC,\Fq)
\end{equation}
is nothing but $\Gamma_{\F_q/\F_q}^{\sX}$. Using Prop. \ref{sheafification of linear
functors}, we conclude the proof by the following lemma.\end{proof}
\begin{lem}
    The functor $\Omega \colon \QCoh_f(\drin(\stX))\to
    \L_{\F_q}(\shC,\Fq)$ is fully faithful and its essential image is the same
    as that of $\QCoh_f(\sX)$ under $\Gamma_{\F_q/\F_q}^{\sX}$.
\end{lem}
\begin{proof}
 Let $\shF=(\shF,\sigma),\,\shG=(\shG,\delta)\in \QCoh_f(\drin(\stX))$. Since $\Omega_{(\sF,\sigma)}\otimes_\Fq k \xrightarrow{\cong} 
 \Gamma_\shF$, $\Omega_{(\sF,\sigma)}$ is left-exact, by Prop. \ref{sheafification of linear
 functors} there is $\sF_0\in\QCoh(\sX)$ such that
 $\Omega_{(\shF,\sigma)}\cong\Gamma_{\sF_0}$. The fact that
 \[\Gamma_\shF\cong\Omega_{(\sF,\sigma)}\otimes_\Fq k
     \cong\Gamma_{\sF_0}\otimes_{\F_q}k\cong
 \Gamma_{\sF_0\otimes_{\F_q}k} 
\] implies that $\sF_0\in\QCoh_f(\sX)$. This proves the second claim. 
 
 In view of Prop. \ref{sheafification of linear functors} the functor
 \(
     \Gamma_{k/\F_q}^{\sX}\colon \QCoh(\stX_k)\to \L_\Fq(\shC,k)
 \)
 is fully faithful. By construction $\Omega$ is also faithful. 
  Given a map $\beta\colon \Omega_{(\shF,\sigma)}\to \Omega_{(\shG,\delta)}$,
  we may interpret it, via the equivalence $\Vect(\Fq) \xrightarrow{\simeq}
  \Vect(\drin(\Fq))$, as a map $\beta_k\colon \Gamma_\shF\to \Gamma_\shG$
  such that
  $\beta_k\circ\Gamma_{\sigma}=\Gamma_{\delta}\circ\phi_k^*\beta_k$.
     \[
  \begin{tikzpicture}[xscale=3.0,yscale=-1.2]
    \node (A0_0) at (0, 0) {$\Gamma_{\phi_k^*\shF}=\phi_k^*\Gamma_\shF$};
    \node (A0_1) at (1, 0) {$\Gamma_\shF$};
    \node (A1_0) at (0, 1) {$\Gamma_{\phi_k^*\shG}=\phi_k^*\Gamma_\shG$};
    \node (A1_1) at (1, 1) {$\Gamma_\shG$};
    \path (A0_0) edge [->]node [auto] {$\scriptstyle{\Gamma_\sigma}$} (A0_1);
    \path (A0_0) edge [->]node [auto]
    {$\scriptstyle{\phi_k^*\beta_k}$} (A1_0);
    \path (A0_1) edge [->]node [auto]
    {$\scriptstyle{\beta_k}$} (A1_1);
    \path (A1_0) edge [->]node [auto] {$\scriptstyle{\Gamma_\delta}$} (A1_1);
  \end{tikzpicture}
  \]
By Prop. \ref{sheafification of linear functors}, $\beta_k=\Gamma_\zeta$
for a unique $\zeta\colon \shF\to \shG$. Moreover,
$\zeta\circ\sigma=\delta\circ \phi_k^*\zeta$ as required.
\end{proof}

\section{The etal\'e case}

In this section we consider the case when $\stM=\Et$ is the
stack over $\Fq$ of representable, \'etale, quasi-compact and quasi-separated
morphisms from algebraic spaces. This is an fppf-stack
(cf.~\cite[\href{https://stacks.math.columbia.edu/tag/0ADV}{0ADV}]{stacks-project}). Let $\stM=\Et_s$ denote
the substack of $\Et$ of \'etale maps which are separated. This is an
fpqc-stack
(cf.~\cite[\href{https://stacks.math.columbia.edu/tag/0246}{0246},
\href{https://stacks.math.columbia.edu/tag/02LR}{02LR}]{stacks-project}). The goal is to prove Theorem \ref{etale separated thm}, (C).
We use the abbreviation \textit{qcqs} for
    quasi-compact and quasi-separated, and \textit{qcs} for
    quasi-compact and separated.

If $\stX$ is a category fibered in groupoids then $\Et(\stX)$
(resp. $\Et_s(\sX)$) is
the category of maps $\stV\to \stX$ which are representable,
\'etale and qcqs (resp. qcs). Arrows are any morphisms
between such objects: $\Et$ and $\Et_s$ are not fibered in groupoids.
Since locally quasi-finite and separated morphisms of algebraic
spaces are schematically representable (see
\cite[\href{https://stacks.math.columbia.edu/tag/03XX}{03XX}]{stacks-project}),
it follows that if $V\to X$ is an object of $\Et_s(X)$ and $X$ is a scheme then also $V$ is a scheme.

We denote by $\textup{FP}$ the fibered category over $\Z$ of finitely presented (in particular qcqs) morphisms from algebraic spaces: if $\stX$ is a category fibered in gropoids then $\textup{FP}(\stX)=\Hom(\stX,\textup{FP})$ is the category of representable morphisms $\stY\to \stX$ which are finitely presented.
\begin{prop}\label{limits and separated}
   Let $A=\varinjlim_i A_i$ be a filtered direct limit of rings.
Then the natural functor
 \[
 \varinjlim_i \textup{FP}(A_i) \arr \textup{FP}(A)
 \]
 is an equivalence and it preserves all the following properties of morphisms: schematic, affine, separated, smooth, \'etale.
 \end{prop}
\begin{proof}
The first claim is
\cite[\href{https://stacks.math.columbia.edu/tag/07SK}{07SK}]{stacks-project}. For the other ones we have instead: schematic by  \cite[\href{https://stacks.math.columbia.edu/tag/01ZM}{01ZM}]{stacks-project}; affine by  \cite[\href{https://stacks.math.columbia.edu/tag/01ZN}{01ZN}]{stacks-project}; separated by  \cite[\href{https://stacks.math.columbia.edu/tag/0851}{0851}]{stacks-project}; smooth by  \cite[\href{https://stacks.math.columbia.edu/tag/0CN2}{0CN2}]{stacks-project}; \'etale by \cite[\href{https://stacks.math.columbia.edu/tag/07SL}{07SL}]{stacks-project}.
\end{proof}

\begin{proof}[Proof of Theorem \ref{etale separated thm}, {\rm(C)}]

By Lemma \ref{from schemes to fibered categories} we can assume that $\stX=X=\Spec A$ is affine.
 An object of $\Et_s(\drin(A))$ is a pair $(V,\sigma)$ where
     $V\to \Spec(A_k)$ is an étale and qcs map of schemes and $\sigma\colon V\to V$ is an isomorphism making the following diagram commutative
   \[
  \begin{tikzpicture}[xscale=2.9,yscale=-1.2]
    \node (A0_0) at (0, 0) {$V$};
    \node (A0_1) at (1, 0) {$V$};
    \node (A1_0) at (0, 1) {$\Spec(A_k)$};
    \node (A1_1) at (1, 1) {$\Spec(A_k)$};
    \path (A0_0) edge [->]node [auto] {$\scriptstyle{\sigma}$} (A0_1);
    \path (A1_0) edge [->]node [auto] {$\scriptstyle{\phi_k}$} (A1_1);
    \path (A0_1) edge [->]node [auto] {$\scriptstyle{}$} (A1_1);
    \path (A0_0) edge [->]node [auto] {$\scriptstyle{}$} (A1_0);
  \end{tikzpicture}
  \]

 Let $\{A_i\}_i$ be the set of all $\Fq$-subalgebras of $A$ which are of finite type over $\Fq$, so that $A$ is a filtered direct limit of the $A_i$. We have that
 \[
     \varinjlim_i \Et_s(A_i) \to \Et_s(A)\hspace{20pt}
     \text{and}\hspace{20pt} \varinjlim_i \Et_s((A_i)_k) \to
     \Et_s(A_k)
 \]
 are equivalences thanks to Prop. \ref{limits and separated}.  Therefore  the functor
 \[
 \varinjlim_i \Et_s(\drin(A_i)) \to \Et_s(\drin(A))
 \]
 is an equivalence. In particular we can assume that $A$ is of finite type over $\Fq$.
 
 Let $\overline X$ be any compactification of $X=\Spec A$, so that $X$ is an open subset of $\overline X$. Composing along $X\to \overline X$ we get a commutative diagram
   \[
  \begin{tikzpicture}[xscale=2.9,yscale=-1.2]
    \node (A0_0) at (0, 0) {$\Et_s(X)$};
    \node (A0_1) at (1, 0) {$\Et_s(\drin(X))$};
    \node (A1_0) at (0, 1) {$\Et_s(\overline X)$};
    \node (A1_1) at (1, 1) {$\Et_s(\drin(\overline X))$};
    \path (A0_0) edge [->]node [auto] {$\scriptstyle{}$} (A0_1);
    \path (A0_0) edge [->]node [auto] {$\scriptstyle{}$} (A1_0);
    \path (A0_1) edge [->]node [auto] {$\scriptstyle{}$} (A1_1);
    \path (A1_0) edge [->]node [auto] {$\scriptstyle{}$} (A1_1);
  \end{tikzpicture}
  \]
The vertical functor are fully faithful. It is easy to see that the result for $\overline X$ implies the result for $X$. Thus we can assume that $X$ is a projective scheme over
$\Fq$. Since the small \'etale site does not change under a
nilpotent closed immersion, we can moreover assume that $X$ is
reduced. In conclusion we can assume that $X$ is a projective
variety over $\Fq$.

By
\cite[\href{https://stacks.math.columbia.edu/tag/02LR}{02LR}]{stacks-project}
and
\cite[\href{https://stacks.math.columbia.edu/tag/03GR}{03GR}]{stacks-project}
 if $Z$ is a reduced scheme of finite type over some field (hence Nagata) and $V\to Z$ is a map of schemes \'etale and separated then the normalization $ N(V)\to Z$ of $Z$ in $V$ is finite and $V\to N(V)$ is an open immersion.

 Let $\textup{Fin}$ denote the stack of finite morphisms.
 By the functoriality of normalizations
 \cite[\href{https://stacks.math.columbia.edu/tag/035J}{035J}]{stacks-project}
 we get a functor
 $$\Et_s(X)\arr \textup{Fin}(X)$$
 Again by functoriality  \cite[\href{https://stacks.math.columbia.edu/tag/035J}{035J}]{stacks-project}, if $(V\to X_k,\sigma)\in \Et_s(\drin(X))$, then
 $\sigma\colon V\to V$ extends to an \textit{isomorphism} $N(\sigma)\colon
 N(V)\to N(V)$ over $\phi_k \colon X_k\to X_k$. Thus we get a morphism $\Et_s(\drin(X))\to \textup{Fin}(\drin(X))$. We claim that the following diagram is commutative
   \[
  \begin{tikzpicture}[xscale=4,yscale=-1.2]
    \node (A0_0) at (0, 0) {$\Et_s(X)$};
    \node (A0_1) at (1, 0) {$\Et_s(\drin(X))$};
    \node (A1_0) at (0, 1) {$\textup{Fin}(X)$};
    \node (A1_1) at (1, 1) {$\textup{Fin}(\drin(X))$};
    \path (A0_0) edge [->]node [auto] {$\scriptstyle{\alpha_{\Et_s,X}}$} (A0_1);
    \path (A0_0) edge [->]node [auto] {$\scriptstyle{}$} (A1_0);
    \path (A0_1) edge [->]node [auto] {$\scriptstyle{}$} (A1_1);
    \path (A1_0) edge [->]node [auto] {$\scriptstyle{}$} (A1_1);
  \end{tikzpicture}
  \]
We are going to use that the bottom functor is an equivalence thanks to Theorem \ref{special X}, (1). Given $U \to X$ in $\Et_s(X)$ we can construct the following Cartesian diagrams:
   \[
  \begin{tikzpicture}[xscale=2.1,yscale=-1.2]
    \node (A0_0) at (0, 0) {$U_k$};
    \node (A0_1) at (1, 0) {$N(U_k)$};
    \node (A0_2) at (2, 0) {$N(U)_k$};
    \node (A0_3) at (3, 0) {$X_k$};
    \node (A1_0) at (0, 1) {$W$};
    \node (A1_1) at (1, 1) {$N$};
    \node (A1_2) at (2, 1) {$N(U)$};
    \node (A1_3) at (3, 1) {$X$};
    \path (A0_0) edge [->]node [auto] {$\scriptstyle{}$} (A0_1);
    \path (A0_1) edge [->]node [auto] {$\scriptstyle{}$} (A1_1);
    \path (A1_0) edge [->]node [auto] {$\scriptstyle{}$} (A1_1);
    \path (A0_3) edge [->]node [auto] {$\scriptstyle{}$} (A1_3);
    \path (A1_1) edge [->]node [auto] {$\scriptstyle{u}$} (A1_2);
    \path (A0_2) edge [->]node [auto] {$\scriptstyle{}$} (A1_2);
    \path (A0_0) edge [->]node [auto] {$\scriptstyle{}$} (A1_0);
    \path (A0_1) edge [->]node [auto] {$\scriptstyle{v}$} (A0_2);
    \path (A1_2) edge [->]node [auto] {$\scriptstyle{}$} (A1_3);
    \path (A0_2) edge [->]node [auto] {$\scriptstyle{}$} (A0_3);
  \end{tikzpicture}
  \]
where the top arrows should be thought of morphisms over $\drin(X)$.
The finite map $N\to X$ and the morphism $u$ are obtained using that $\alpha_{\textup{Fin},X}$ is an equivalence. The map $W\to N$ is the open immersion obtained from Theorem \ref{the case of closed immersion}. Since $W\to N(U)$ pullback to an open immersion, it is an open immersion. Applying again Theorem \ref{the case of closed immersion} we can conclude that $W=U\to N(U)$ is the given open immersion.

By the universal property of normalization we can conclude that $N\to N(U)$ is an isomorphism.

We now show that $\alpha_{\Et_s,X}$ is an equivalence.

For the fully faithfulness, given $U\to X$ and $W\to X$ in $\Et_s(X)$ and a morphism $\drin(U)\to \drin(W)$, this map extends to a morphism $\drin(N(U))\to \drin(N(W))$ induced by a map $a\colon  N(U)\to N(W)$. We have $a(U)\subseteq W$ because this relation holds after pulling back to $N(W)_k$.

For the essential surjectivity, starting
 with $(V\to X_k,\sigma)\in \Et_s(\drin(X))$ we get
 $(N(V)\to X_k,N(\sigma))\in\textup{Fin}(\drin(X))$ which descents to
 $(N\to X) \in\textup{Fin}(X)$. Then we get
  the descent $U\subseteq N$ of the open embedding $V\subseteq
 N(V)=N_k$ by Theorem \ref{the case of closed immersion}.
\end{proof}

\begin{prop}
Let $\stX$ be a category fibered in groupoids over $\Fq$. Then the functor
\begin{equation}\label{functor for Et}
 \Et(\stX)\longrightarrow \Et(\drin(\stX))
\end{equation}
 is fully faithful.
\end{prop}
\begin{proof}
    According to Lemma \ref{from schemes to fibered categories} we may assume that
    $\sX=X=\Spec A$. 
    
    Let $Y\to X$ be an object of $\Et(X)$, that is an \'etale
    and qcqs morphism of algebraic spaces and $U\to Y$ be a map from an algebraic space. If $(U\to X)\in \Et_s(X)$, then $(U\to Y)\in \Et_s(Y)$. Indeed $U\to Y$ is quasi-compact because $Y$ is quasi-separated and it is separated by \cite[\href{https://stacks.math.columbia.edu/tag/03KR}{03KR}]{stacks-project}.
    
    If $U\to Y$ is an \'etale atlas from a qcs scheme and $R=U\times_Y U$, then $R$ is qcs because $Y$ is quasi-separated and $R\to U\times_X U$ is a monomorphism. In particular $R\rightrightarrows U$ defines a groupoid in $\Et_s(X)$ but also in $\Et_s(Y)$.
    
    We come back to the problem of fully faithfulness. Faithfulness follows from Prop. \ref{main functor is faithful}.
    
    Suppose that $Z\arr X$ is another
    object in $\Et(X)$, and suppose that $\lambda\colon
    Y_k\arr Z_k$ comes from a morphism in $\Et(\drin(X))$. Let $R\rightrightarrows U\arr Y$ be an \'etale
    presentation by qcs \'etale schemes. A map $Y\arr Z$ is the
    same as a map $U\arr Z$ which equalizes $R\rightrightarrows
    U$. By Theorem \ref{etale separated thm}, (C) we        
    may assume that $Y$ is a qcs \'etale scheme over $X$.
    
    Now suppose that $R\rightrightarrows U\arr Z$ is an \'etale
    presentation by qcs \'etale schemes. Then the pullback
    $$R_k'\rightrightarrows U_k'\arr Y_k$$
    of
    $R_k\rightrightarrows U_k\arr Z_k$ along $\lambda$ is an
    \'etale presentation of $Y_k$. Moreover this presentation belongs to $\Et_s(Y_k)$. More precisely it defines a groupoid in $\Et_s(\drin(Y))$.
    By Theorem \ref{etale separated thm}, (C) it descend
    to an \'etale presentation $$R'\rightrightarrows U'\arr Y$$
    in $\Et_s(Y)$. On the other hand, since $Y\to X$ is separated, it follows that the presentation actually belongs to $\Et_s(X)$. 
    
    By construction the original morphism $\lambda\colon Y_k\to Z_k$ defines a morphism of the groupoid presentations in $\Et_s(\drin(X))$. By Theorem \ref{etale separated thm}, (C) this morphism induces a morphism of the corresponding groupoid presentations in $\Et_s(X)$. By
    \'etale
    descent we get a map $Y\arr Z$ inducing $\lambda$.
\end{proof}

\section{The non proper case}

The main of goal of this section is to finish the proof of Theorem \ref{special M}.

\begin{proof}[Proof of Theorem \ref{special M}, {\rm(A)}]
Let $\stM$ be a Deligne Mumford stack over $\Fq$ with qcs diagonal. In particular, by \ref{fully faithful for separated diagonal}, we know that the map $\alpha_{\stM,\stV}$ is fully faithful for all categories fibered in groupoids $\stV$ over $\Fq$.
So we are interested in the essential surjectivity.
 
 By \ref{from schemes to fibered categories} we can assume that $\stX=X=\Spec A$ is affine. 
Let $\drin(X)\to \stM$ be a map. 
 Since $\stM$ is a union of quasi-compact open substacks, we see
 that $\drin(X)\to \stM$ factors through one of this opens. In
 other words we can assume that $\stM$ is quasi-compact, that is
 there exists an \'etale map $W\to \stM$ where $W$ is an affine
 scheme. The condition on the diagonal of $\stM$ assures that
 $W\to \stM$ is an \'etale and qcs map. Putting together Theorem \ref{etale separated thm}, (C), \ref{essential surjectivity main lemma} with $B=W$ and \ref{the affine case} we obtain the desired map $X\to \stM$.
\end{proof}

\begin{proof}[Proof of Theorem \ref{special M}, {\rm(2)}]
 By \ref{from schemes to fibered categories} we can assume that $\stX=X=\Spec A$
 is affine. Let $\xi,\eta\in \stM(X)$, set
 $I=\Isosh_\stM(\xi,\eta)\to X$ and let $\drin(X)\to I$ be a
 map. Since $I$ is a quasi-separated algebraic space, by Theorem \ref{special M}, (B) we find a factorization $\drin(X)\to X \to I$. Since $\drin(X)\to I \to X$ is the canonical map, again Theorem \ref{special M}, (B) tells us that $X\to I \to X$ is the identity, as required. 
\end{proof}

\section{Applications}

\subsection{Affine gerbes over $\Fq$ are trivial}

\begin{thm}\label{affine gerbes are trivial}
 Affine gerbes over $\Fq$ are trivial.
\end{thm}
\begin{proof}
 We apply Theorem \ref{special X} with $\stX=\Spec \Fq$ and $\stM$ an affine gerbe. Choose an algebraically closed field $k$ such that there is an object $\xi\in \stM(k)$. It could be that $\xi$ and $\phi_k^*\xi$ are not isomorphic, but they surely become so enlarging the field $k$. Thus we get an object of $\stM(\drin(\Fq))\simeq \stM(\Fq)$.
\end{proof}

\subsection{Drinfeld's Lemma} Let $\sX$ be any connected algebraic stack over $\F_q$ with a geometric point $x$. According to \cite[\nopp\S 4]{Noohi04}, the category $\FEt(\sX)$ is a Galois category, and we denote by $\pi_1^\et(\sX,x)$ its étale fundamental group. An immediate consequence of Theorem \ref{etale separated thm} is that the pullback functor induces an equivalence
\(
    \FEt(\sX)\xlongrightarrow{\simeq} \FEt(\drin(\sX)).
\) This implies:
\begin{cor}\label{permanence}
    If $\sX$ is a connected algebraic stack over $\F_q$ with a geometric point $x$, then for any algebraically closed field $k$ containing $\F_q$, and for any geometric point $y$ of $\drin(\sX)$ lifting $x$, there is a natural isomorphism \[\pi_1^\et(\drin(\sX),y)\cong \pi_1^\et(\sX,x)\]
\end{cor}

Note that by Theorem \ref{the case of closed immersion}, the clopen subspaces of $\drin(\sX)$ is in one-to-one correspondence with the clopen subspaces of $\sX$. Thus that $\sX$ is connected implies that $\drin(\sX)$ is connected. This justifies the notations in Theorem \ref{permanence}.

In the language of \cite[Lecture 16]{SW20}, one can say that the space $\drin(\sX)$ satisfies \textit{permanence of $\pi_1$} under algebraically base field extensions. Starting from this, there is a standardized procedure (cf.~\cite[\nopp Lecture 16]{SW20}, \cite[\nopp \S 4.2]{Ked17}) to conclude the following \emph{Drinfeld-Lau descent}.

\begin{thm}[Drinfeld's Lemma]
    Let $\sX_1,\sX_2,\dots,\sX_n$ be connected algebraic stacks over $\F_q$, and set $\sX\coloneqq \sX_1\times_{\F_q}\sX_2\times_{\F_q}\cdots\times_{\F_q}\sX_n$. Choose a geometric point $x$ in $\sX$, and denote $x_i$ the image of $x$ in $\sX$. Let $\phi_{i}$ be the $q$-th relative Frobenius of $\sX_i$. Then one obtains the category $\FEt(\sX/\Phi)$ of finite étale covers of $\sX$ equipped with partial Frobenius actions just as in {\rm \cite[Def. 16.2.1]{SW20}}. Then $\FEt(\sX/\Phi)$ is a Galois category, and the base change along $x$ functor induces a fiber functor. Moreover, the natural map
    \[
\pi_1^\et(\sX/\Phi,x)\longrightarrow \pi_1^\et(\sX_1,x_1)\times\cdots\times\pi_1^\et(\sX_n,x_n)
    \] is an isomorphism.
\end{thm}

A detailed proof of this result will be long, so we'll move it to a separated paper.

\subsection{The pro-étale fundamental group}\label{pro-étale}

In \cite{BS15}, B.~Bhatt and P.~Scholze introduced the notion of \emph{geometric covers} of a scheme $X$. A map of schemes $f\colon Y\to X$ is a geometric cover if it is étale and satisfies the valuative criterion for properness. A geometric cover is a finite étale cover if and only if it is quasi-compact. B.~Bhatt and P.~Scholze showed that if $X$ is a topologically locally Noetherian scheme with a geometric point $x$, then the category $\Cov(X)$ of geometric covers of $X$ together with the base change along $x$ functor form an infinite Galois category, therefore, they correspond, via the infinite Galois theory, to a Noohi group $\pi_1^\pet(X,x)$ which is called the pro-étale fundamental group of $X$ at $x$.

Here we want to point out that geometric covers (resp. pro-étale fundamental groups) do not satisfy Drinfeld-Lau descent (resp. Drinfeld's lemma) even for the simplest $\F_q$-scheme $X=\Spec(\F_q)$. Indeed, since both the étale maps and the valuative criterion satisfy fpqc descent (or more generally, universally submersive descent, cf.~\cite{HS23}), the geometric covers satisfies fpqc descent (resp. universally submersive descent). This implies that $\Cov(\drin(\F_q))$ is equivalent to the category consisting of objects of $\Cov(k)$ equipped with $\Z$-actions. However, since $k$ is algebraically closed, $\Cov(k)$ is nothing but the category of sets. By the infinite Galois category theory, the pro-étale fundamental group $\pi_1^\pet(\drin(\F_q))=\Z$, which is not equal to the fundamental group $\pi_1^{\pet}(\F_q)=\pi_1^\et(\F_q)=\hat{\Z}$.        

\printbibliography
\end{document}

%% file: packages_and_functions.for_preamble.tex
\usepackage{mathabx}
\usepackage{bbm}
\usepackage{fontenc}
\usepackage[T1]{fontenc}
\usepackage{amsfonts}
\usepackage{mathrsfs}
\usepackage{stmaryrd}
\usepackage{upgreek}
\usepackage{textcomp}

\usepackage{extarrow}
\usepackage{enumerate}
\usepackage{graphicx}
\usepackage[colorlinks=true, linkcolor=blue]{hyperref}
\usepackage{latexsym}
\usepackage{mathtools}
\usepackage{mathtext}
\usepackage{amsmath}
\usepackage{amssymb}
\usepackage{amsthm}
\usepackage{tikz}
\usepackage[all,cmtip]{xy}
\usepackage[mathscr]{eucal}
\usepackage[toc,page]{appendix}
\usepackage{a4wide}

\usetikzlibrary{arrows}
\usetikzlibrary{matrix}
\usetikzlibrary{shapes}
\usetikzlibrary{snakes}
\usetikzlibrary{matrix}

\DeclareMathOperator{\Bi}{\textup{B}}

\DeclareMathOperator{\Coh}{\textup{Coh}}

\DeclareMathOperator{\FEt}{\textup{FÉt}}

\DeclareMathOperator{\Homsh}{\underline{\textup{Hom}}}

\DeclareMathOperator{\Isosh}{\underline{\textup{Iso}}}

\DeclareMathOperator{\QCoh}{\textup{QCoh}}

\DeclareMathOperator{\Spec}{\textup{Spec}}

\DeclareMathOperator{\et}{\textup{et}}

%% file: packages_and_functions.tex
\global\long\def\A{\mathbb{A}}

\global\long\def\Ab{(\textup{Ab})}

\global\long\def\C{\mathbb{C}}

\global\long\def\Cat{(\textup{Cat})}

\global\long\def\Di#1{\textup{D}(#1)}

\global\long\def\E{\mathbb{E}}

\global\long\def\F{\mathbb{F}}

\global\long\def\GCov{G\textup{-Cov}}

\global\long\def\Gcat{(\textup{Galois cat})}

\global\long\def\Gfsets#1{#1\textup{-fsets}}

\global\long\def\Gm{\mathbb{G}_{m}}

\global\long\def\GrCov#1{\textup{D}(#1)\textup{-Cov}}

\global\long\def\Grp{(\textup{Grps})}

\global\long\def\Gsets#1{(#1\textup{-sets})}

\global\long\def\HCov{H\textup{-Cov}}

\global\long\def\MCov{\textup{D}(M)\textup{-Cov}}

\global\long\def\MHilb{M\textup{-Hilb}}

\global\long\def\N{\mathbb{N}}

\global\long\def\PGor{\textup{PGor}}

\global\long\def\PGrp{(\textup{Profinite Grp})}

\global\long\def\PP{\mathbb{P}}

\global\long\def\Pj{\mathbb{P}}

\global\long\def\Q{\mathbb{Q}}

\global\long\def\RCov#1{#1\textup{-Cov}}

\global\long\def\RR{\mathbb{R}}

\global\long\def\Sch{\textup{Sch}}

\global\long\def\WW{\textup{W}}

\global\long\def\Z{\mathbb{Z}}

\global\long\def\acts{\curvearrowright}

\global\long\def\alA{\mathscr{A}}

\global\long\def\alB{\mathscr{B}}

\global\long\def\arr{\longrightarrow}

\global\long\def\arrdi#1{\xlongrightarrow{#1}}

\global\long\def\catC{\mathscr{C}}

\global\long\def\catD{\mathscr{D}}

\global\long\def\catF{\mathscr{F}}

\global\long\def\catG{\mathscr{G}}

\global\long\def\comma{,\ }

\global\long\def\covU{\mathcal{U}}

\global\long\def\covV{\mathcal{V}}

\global\long\def\covW{\mathcal{W}}

\global\long\def\duale#1{{#1}^{\vee}}

\global\long\def\fasc#1{\widetilde{#1}}

\global\long\def\fsets{(\textup{f-sets})}

\global\long\def\iL{r\mathscr{L}}

\global\long\def\id{\textup{id}}

\global\long\def\la{\langle}

\global\long\def\odi#1{\mathcal{O}_{#1}}

\global\long\def\ra{\rangle}

\global\long\def\set{(\textup{Sets})}

\global\long\def\sets{(\textup{Sets})}

\global\long\def\shA{\mathcal{A}}

\global\long\def\shB{\mathcal{B}}

\global\long\def\shC{\mathcal{C}}

\global\long\def\shD{\mathcal{D}}

\global\long\def\shE{\mathcal{E}}

\global\long\def\shF{\mathcal{F}}

\global\long\def\shG{\mathcal{G}}

\global\long\def\shH{\mathcal{H}}

\global\long\def\shI{\mathcal{I}}

\global\long\def\shJ{\mathcal{J}}

\global\long\def\shK{\mathcal{K}}

\global\long\def\shL{\mathcal{L}}

\global\long\def\shM{\mathcal{M}}

\global\long\def\shN{\mathcal{N}}

\global\long\def\shO{\mathcal{O}}

\global\long\def\shP{\mathcal{P}}

\global\long\def\shQ{\mathcal{Q}}

\global\long\def\shR{\mathcal{R}}

\global\long\def\shS{\mathcal{S}}

\global\long\def\shT{\mathcal{T}}

\global\long\def\shU{\mathcal{U}}

\global\long\def\shV{\mathcal{V}}

\global\long\def\shW{\mathcal{W}}

\global\long\def\shX{\mathcal{X}}

\global\long\def\shY{\mathcal{Y}}

\global\long\def\shZ{\mathcal{Z}}

\global\long\def\st{\ | \ }

\global\long\def\stA{\mathcal{A}}

\global\long\def\stB{\mathcal{B}}

\global\long\def\stC{\mathcal{C}}

\global\long\def\stD{\mathcal{D}}

\global\long\def\stE{\mathcal{E}}

\global\long\def\stF{\mathcal{F}}

\global\long\def\stG{\mathcal{G}}

\global\long\def\stH{\mathcal{H}}

\global\long\def\stI{\mathcal{I}}

\global\long\def\stJ{\mathcal{J}}

\global\long\def\stK{\mathcal{K}}

\global\long\def\stL{\mathcal{L}}

\global\long\def\stM{\mathcal{M}}

\global\long\def\stN{\mathcal{N}}

\global\long\def\stO{\mathcal{O}}

\global\long\def\stP{\mathcal{P}}

\global\long\def\stQ{\mathcal{Q}}

\global\long\def\stR{\mathcal{R}}

\global\long\def\stS{\mathcal{S}}

\global\long\def\stT{\mathcal{T}}

\global\long\def\stU{\mathcal{U}}

\global\long\def\stV{\mathcal{V}}

\global\long\def\stW{\mathcal{W}}

\global\long\def\stX{\mathcal{X}}

\global\long\def\stY{\mathcal{Y}}

\global\long\def\stZ{\mathcal{Z}}

\global\long\def\then{\ \Longrightarrow\ }

\global\long\def\L{\textup{L}}

\global\long\def\l{\textup{l}}

\newcommand{\B}{{\mathbb B}}
\newcommand{\D}{{\mathbb D}}
\newcommand{\G}{{\mathbb G}}
\renewcommand{\H}{{\mathbb H}}
\newcommand{\I}{{\mathbb I}}
\newcommand{\J}{{\mathbb J}}
\newcommand{\M}{{\mathbb M}}
\renewcommand{\P}{{\mathbb P}}
\newcommand{\R}{{\mathbb R}}
\newcommand{\T}{{\mathbb T}}
\newcommand{\U}{{\mathbb U}}
\newcommand{\V}{{\mathbb V}}
\newcommand{\W}{{\mathbb W}}
\newcommand{\X}{{\mathbb X}}
\newcommand{\Y}{{\mathbb Y}}

\newcommand{\sA}{{\mathcal A}}
\newcommand{\sB}{{\mathcal B}}
\newcommand{\sC}{{\mathcal C}}
\newcommand{\sD}{{\mathcal D}}
\newcommand{\sE}{{\mathcal E}}
\newcommand{\sF}{{\mathcal F}}
\newcommand{\sG}{{\mathcal G}}
\newcommand{\sH}{{\mathcal H}}
\newcommand{\sI}{{\mathcal I}}
\newcommand{\sJ}{{\mathcal J}}
\newcommand{\sK}{{\mathcal K}}
\newcommand{\sL}{{\mathcal L}}
\newcommand{\sM}{{\mathcal M}}
\newcommand{\sN}{{\mathcal N}}
\newcommand{\sO}{{\mathcal O}}
\newcommand{\sP}{{\mathcal P}}
\newcommand{\sQ}{{\mathcal Q}}
\newcommand{\sR}{{\mathcal R}}
\newcommand{\sS}{{\mathcal S}}
\newcommand{\sT}{{\mathcal T}}
\newcommand{\sU}{{\mathcal U}}
\newcommand{\sV}{{\mathcal V}}
\newcommand{\sW}{{\mathcal W}}
\newcommand{\sX}{{\mathcal X}}
\newcommand{\sY}{{\mathcal Y}}
\newcommand{\sZ}{{\mathcal Z}}


\newcommand{\Aff}{{\rm Aff}}
\newcommand{\Aut}{{\rm Aut}}
\newcommand{\an}{{\rm an}}
\newcommand{\Bd}{{\rm Band}}
\newcommand{\Cats}{{\rm Cats}}
\newcommand{\ch}{\textup{Ch}}
\newcommand{\Char}{{\rm char}}
\newcommand{\codim}{{\rm codim}}
\newcommand{\cont}{{\rm cont}}
\newcommand{\Cov}{\textup{Cov}}
\newcommand{\Crys}{{\rm Crys}}
\newcommand{\cts}{\textup{cts}}
\newcommand{\Div}{{\rm Div}}
\newcommand{\Dmod}{{\rm Dmod}}
\newcommand{\ECov}{{\rm ECov}}
\newcommand{\ed}{{\rm ed}}
\newcommand{\Ess}{{\rm EFin}}
\renewcommand{\et}{\textup{\'et}}
\newcommand{\ev}{\textup{ev}}
\newcommand{\Fdiv}{{\rm Fdiv}}
\newcommand{\Fib}{{\rm Fib}}
\newcommand{\FSet}{{\rm FSet}}
\newcommand{\FtAff}{{\rm FtAff}}
\newcommand{\Gal}{{\rm Gal}}
\newcommand{\height}{\textup{ht}}
\newcommand{\Hom}{{\rm Hom}}
\newcommand{\iinf}{\textup{inf}}
\newcommand{\im}{{\rm im}}
\newcommand{\Ker}{{\rm Ker}}
\newcommand{\LL}{\textup{L}}
\newcommand{\Loc}{{\rm Loc}}
\newcommand{\Max}{{\rm Max \ }}
\newcommand{\MIC}{\mbox{MIC}}
\newcommand{\Min}{{\rm Min \ }}
\newcommand{\NN}{\textup{N}}
\newcommand{\Mod}{\text{\sf Mod}}
\newcommand{\Noohi}{\textup{Noohi}}
\newcommand{\perf}{\textup{perf}}
\newcommand{\pet}{{\textup{proét}}}
\newcommand{\Pic}{{\rm Pic}}
\newcommand{\Rep}{\text{\sf Rep}}
\newcommand{\Res}{{\rm Res}}
\newcommand{\rank}{{\rm rank}}
\newcommand{\red}{{\rm red}}
\newcommand{\Sets}{\textup{Sets}}
\newcommand{\Spf}{\textup{Spf}}
\newcommand{\spe}{\textup{sp}}
\newcommand{\str}{\textup{str}}
\newcommand{\strat}{{\rm Str}}
\newcommand{\sym}{\text{Sym}}
\newcommand{\tp}{{\rm top}}
\newcommand{\Tr}{{\rm Tr}}
\newcommand{\trace}{{\rm Tr}}
\newcommand{\vect}{\text{\sf vect}}
\newcommand{\Vect}{\text{\sf Vect}}

%% file: biblio.bib
@article {BS15,
    AUTHOR = {Bhatt, Bhargav and Scholze, Peter},
     TITLE = {The pro-\'{e}tale topology for schemes},
   JOURNAL = {Ast\'{e}risque},
  FJOURNAL = {Ast\'{e}risque},
    NUMBER = {369},
      YEAR = {2015},
     PAGES = {99--201},
      ISSN = {0303-1179},
      ISBN = {978-2-85629-805-3},
   MRCLASS = {14F05 (14F20 14F35 14H30 18B25)},
  MRNUMBER = {3379634},
MRREVIEWER = {Pieter Belmans},
}

@Article{BV15,
    Author = {Niels {Borne} and Angelo {Vistoli}},
    Title = {{The Nori fundamental gerbe of a fibered category.}},
    FJournal = {{Journal of Algebraic Geometry}},
    Journal = {{J. Algebr. Geom.}},
    ISSN = {1056-3911; 1534-7486/e},
    Volume = {24},
    Number = {2},
    Pages = {311--353},
    Year = {2015},
    Publisher = {American Mathematical Society (AMS), Providence, RI; University Press, Chicago, IL},
    Language = {English},
    DOI = {10.1090/S1056-3911-2014-00638-X},
    MSC2010 = {14A20 14F05 18D05},
    Zbl = {1349.14004}
}

@inproceedings {Dri80,
    AUTHOR = {Drinfeld, V. G.},
     TITLE = {Langlands' conjecture for {${\rm GL}(2)$} over functional
              fields},
 BOOKTITLE = {Proceedings of the {I}nternational {C}ongress of
              {M}athematicians ({H}elsinki, 1978)},
     PAGES = {565--574},
 PUBLISHER = {Acad. Sci. Fennica, Helsinki},
      YEAR = {1980},
      ISBN = {951-41-0352-1},
   MRCLASS = {12A67 (10D40 22E55)},
  MRNUMBER = {562656},
MRREVIEWER = {Joe\ Repka},
}

@book {Goss98,
    AUTHOR = {Goss, David},
     TITLE = {Basic structures of function field arithmetic},
    SERIES = {Ergebnisse der Mathematik und ihrer Grenzgebiete (3) [Results
              in Mathematics and Related Areas (3)]},
    VOLUME = {35},
 PUBLISHER = {Springer-Verlag, Berlin},
      YEAR = {1996},
     PAGES = {xiv+422},
      ISBN = {3-540-61087-1},
   MRCLASS = {11G09 (11L05 11R58)},
  MRNUMBER = {1423131},
MRREVIEWER = {Jeremy T. Teitelbaum},
       DOI = {10.1007/978-3-642-61480-4},
       URL = {https://doi.org/10.1007/978-3-642-61480-4},
}

@article {HS23,
    AUTHOR = {Hansen, David and Scholze, Peter},
     TITLE = {Relative perversity},
   JOURNAL = {Comm. Amer. Math. Soc.},
  FJOURNAL = {Communications of the American Mathematical Society},
    VOLUME = {3},
      YEAR = {2023},
     PAGES = {631--668},
      ISSN = {2692-3688},
   MRCLASS = {14F43 (14F08 14F20)},
  MRNUMBER = {4630128},
       DOI = {10.1090/cams/21},
       URL = {https://doi.org/10.1090/cams/21},
}

@misc {Ked17,
author = {Kiran S. {Kedlaya}},
title = {Sheaves, stacks, and shtukas},
URL = {http://swc.math.arizona.edu/aws/2017/2017KedlayaNotes.pdf},
year = 2017,
}

@article {Laf97,
    AUTHOR = {Lafforgue, Laurent},
     TITLE = {Chtoucas de {D}rinfeld et conjecture de
              {R}amanujan-{P}etersson},
   JOURNAL = {Astérisque},
      FJOURNAL = {Astérisque},
    NUMBER = {243},
      YEAR = {1997},
     PAGES = {ii+329},
      ISSN = {0303-1179},
   MRCLASS = {11G09 (11R39)},
  MRNUMBER = {1600006},
MRREVIEWER = {David Goss},
}

@article {Lau07,
    AUTHOR = {Lau, Eike},
     TITLE = {On degenerations of {$\mathcal{D}$}-shtukas},
   JOURNAL = {Duke Math. J.},
  FJOURNAL = {Duke Mathematical Journal},
    VOLUME = {140},
      YEAR = {2007},
    NUMBER = {2},
     PAGES = {351--389},
      ISSN = {0012-7094},
   MRCLASS = {11R58 (11G09 14D20)},
  MRNUMBER = {2359823},
MRREVIEWER = {Alexey A. Panchishkin},
       DOI = {10.1215/S0012-7094-07-14025-0},
       URL = {https://doi.org/10.1215/S0012-7094-07-14025-0},
}

@book {LMB00,
    AUTHOR = {Laumon, G\'{e}rard and Moret-Bailly, Laurent},
     TITLE = {Champs alg\'{e}briques},
    SERIES = {Ergebnisse der Mathematik und ihrer Grenzgebiete. 3. Folge. A
              Series of Modern Surveys in Mathematics [Results in
              Mathematics and Related Areas. 3rd Series. A Series of Modern
              Surveys in Mathematics]},
    VOLUME = {39},
 PUBLISHER = {Springer-Verlag, Berlin},
      YEAR = {2000},
     PAGES = {xii+208},
      ISBN = {3-540-65761-4},
   MRCLASS = {14A20 (14D20)},
  MRNUMBER = {1771927},
MRREVIEWER = {Dan Edidin},
}

@article{Noohi04, 
    title={FUNDAMENTAL GROUPS OF ALGEBRAIC STACKS}, 
    volume={3}, DOI={10.1017/S1474748004000039}, 
    number={1}, 
    journal={Journal of the Institute of Mathematics of Jussieu}, 
    author={Noohi, B.}, 
    year={2004}, 
    pages={69–103}
}

@article {Rydh15,
    AUTHOR = {Rydh, David},
     TITLE = {Noetherian approximation of algebraic spaces and stacks},
   JOURNAL = {J. Algebra},
  FJOURNAL = {Journal of Algebra},
    VOLUME = {422},
      YEAR = {2015},
     PAGES = {105--147},
      ISSN = {0021-8693,1090-266X},
   MRCLASS = {14A20},
  MRNUMBER = {3272071},
MRREVIEWER = {Hsian-Hua\ Tseng},
       DOI = {10.1016/j.jalgebra.2014.09.012},
       URL = {https://doi.org/10.1016/j.jalgebra.2014.09.012},
}

@article {Rydh16,
    AUTHOR = {Rydh, David},
     TITLE = {Approximation of sheaves on algebraic stacks},
   JOURNAL = {Int. Math. Res. Not. IMRN},
  FJOURNAL = {International Mathematics Research Notices. IMRN},
      YEAR = {2016},
    NUMBER = {3},
     PAGES = {717--737},
      ISSN = {1073-7928},
   MRCLASS = {14A20 (14F05)},
  MRNUMBER = {3493431},
MRREVIEWER = {Hsian-Hua Tseng},
       DOI = {10.1093/imrn/rnv142},
       URL = {https://doi.org/10.1093/imrn/rnv142},
}

@book {SW20,
    AUTHOR = {Scholze, Peter and Weinstein, Jared},
     TITLE = {Berkeley lectures on {$p$}-adic geometry},
    SERIES = {Annals of Mathematics Studies},
    VOLUME = {207},
 PUBLISHER = {Princeton University Press, Princeton, NJ},
      YEAR = {2020},
     PAGES = {x+250},
      ISBN = {978-0-691-20209-9; 978-0-691-20208-2; 978-0-691-20215-0},
   MRCLASS = {14G45 (14A15 14G22 14G35 14M15)},
  MRNUMBER = {4446467},
}

@MISC{stacks-project,
    AUTHOR = "Authors, The Stacks Project",
    TITLE = "Stacks Project",
    URL = "https://stacks.math.columbia.edu/"
}

@misc{Ton20,
      title={Sheafification of linear functors}, 
      author={Fabio Tonini},
      year={2020},
      eprint={1409.4073},
      archivePrefix={arXiv},
      primaryClass={math.AG}
}

@incollection {Vis05,
    AUTHOR = {Vistoli, Angelo},
     TITLE = {Grothendieck topologies, fibered categories and descent
              theory},
 BOOKTITLE = {Fundamental algebraic geometry},
    SERIES = {Math. Surveys Monogr.},
    VOLUME = {123},
     PAGES = {1--104},
 PUBLISHER = {Amer. Math. Soc., Providence, RI},
      YEAR = {2005},
   MRCLASS = {14F20 (14A20)},
  MRNUMBER = {2223406},
       DOI = {10.1007/s00222-005-0429-0},
       URL = {https://doi.org/10.1007/s00222-005-0429-0},
}
